\documentclass[12pt,a4paper]{article}
\usepackage[T1]{fontenc}
\usepackage[centertags]{amsmath}
\usepackage{amsfonts}
\usepackage{amssymb}
\usepackage{amsthm}
\usepackage{newlfont}
\usepackage{indentfirst}
\usepackage{float}
\usepackage{enumitem}

\DeclareMathOperator{\tr}{tr}
\DeclareMathOperator{\ch}{char}
\DeclareMathOperator{\GL}{GL}
\DeclareMathOperator{\SL}{SL}
\DeclareMathOperator{\GA}{GA}
\DeclareMathOperator{\SA}{SA}
\DeclareMathOperator{\Conj}{Conj}
\DeclareMathOperator{\Ker}{Ker}

\newtheorem{tw}{{\sf Theorem}}[section]
\newtheorem{lem}[tw]{{\sf Lemma}}
\newtheorem{wn}[tw]{{\sf Corollary}}
\newtheorem{prop}[tw]{{\sf Proposition}}

\theoremstyle{definition}
\newtheorem{deff}[tw]{{\sf Definition}}

\newtheorem{ex}[tw]{{\sf Example}}
\newtheorem{uw}[tw]{{\sf Remark}}

\numberwithin{equation}{section}

\newcommand{\blank}{{\mspace{1mu}\cdot\mspace{1mu}}}

\title{The Banach--Tarski paradox for some subsets of finite-dimensional normed spaces over non-Archimedean valued fields}

\author{Kamil Orzechowski\\
University of Rzesz\'ow\\ 
The Doctoral School\\
Rejtana 16c\\
35-959 Rzesz\'ow, Poland\\
}

\begin{document}

\maketitle

\renewcommand{\thefootnote}{}

\footnote{2020 \emph{Mathematics Subject Classification}: Primary 47S10; Secondary 46S10, 12J25, 26E30, 20E05, 20H20, 46B04, 05A18, 03E25.}

\footnote{\emph{Key words and phrases}: Banach--Tarski paradox, paradoxical decomposition, non-Archimedean valued field, non-Archimedean normed space, free group.}

\begin{abstract}
    We show some results related to the classical Banach--Tarski paradox in the setting of finite-dimensional normed spaces over a non-Archimedean valued field $K$. For instance, all balls and spheres in $K^n$, and the whole space $K^n$ (for $n\ge 2$) are paradoxical with respect to certain groups of isometries of $K^n$.
    If $K$ is locally compact (e.g., $K$ is the field $\mathbb{Q}_p$ of $p$-adic numbers for any prime number $p$), any two bounded subsets of $K^n$ with nonempty interiors are equidecomposable (and paradoxical) with respect to a certain group of isometries of $K^n$ (for $n\ge 2$).
\end{abstract}

\section{Introduction}

In 1924, Banach and Tarski (inspired by a result of Hausdorff \cite[p. 469]{Haus}) proved one of the most shocking theorems in mathematics, called the {\em Banach--Tarski paradox} \cite[Theorem 24]{BT}: 

{\it Any two bounded subsets $A$ and $B$ of $\mathbb{R}^3$ with nonempty interiors can be partitioned into finite number of pieces $A_1, \ldots, A_n$ and $B_1, \ldots, B_n$, respectively, in such a way that for some isometries $\tau_1, \ldots, \tau_n$ of $\mathbb{R}^3$ we have $\tau_k(A_k)=B_k$ for any $1\leq k \leq n.$} 

In modern terms, we can then say that $A$ and $B$ are {\em $G$-equidecomposable}, where $G$ is the group of all isometries of $\mathbb{R}^3$.
It follows that any ball in $\mathbb{R}^3$ can be partitioned into finite number of pieces in such a way that the pieces can be transformed using isometries of $\mathbb{R}^3$ to obtain two balls identical with the initial one \cite[Lemma 21]{BT}, i.e., any ball in $\mathbb{R}^3$ is {\em $G$-paradoxical}. Those results are called paradoxes as they seem counterintuitive. Clearly, the pieces involved cannot all be Lebesgue-measurable. Note that the Banach--Tarski paradox is a theorem in ZFC, but not in ZF \cite[Corollary 15.3]{W}.

 There have been many refinements, generalizations and inspired results since the work of Banach and Tarski was published. They include minimizing the number of pieces needed for a paradoxical decomposition (five is optimal for a ball in $\mathbb{R}^3$), generalization for balls and spheres in $\mathbb{R}^n$, $n\ge 3$, as well as similar results in hyperbolic spaces. The monograph \cite{W} is an excellent source covering most of these topics.

In this paper, we prove an analog of the Banach--Tarski paradox in the setting of finite-dimensional normed spaces $K^n$ over a non-Archimedean valued field $(K,\lvert\blank\rvert)$ instead of $(\mathbb{R}, \lvert\blank\rvert)$, e.g., $(K,\lvert\blank\rvert)$ may be the field $(\mathbb{Q}_p, \lvert\blank\rvert_p)$ of $p$-adic numbers or the field $(\mathbb{C}_p, \lvert\blank\rvert_p)$ of $p$-adic complex numbers for any prime number $p$. We are especially interested in paradoxical decompositions of balls and spheres in $K^n$, the whole space $K^n$, and the set $K^n\setminus \{\mathbf{0}\}$, $n\ge 2$.

The paper is organized as follows. Section 2 provides basic definitions of the theory of non-Archimedean valued fields and normed spaces over them. A characterization of linear and affine isometries of $K^n$ is given and some notation is introduced. In Section 3, for a group $G$ acting on a set $X$ we recall the notion of a {\em G-paradoxical} set $E\subseteq X$, and some basic results; in particular: If a group $G$ contains a non-Abelian free subgroup and acts without nontrivial fixed points on a set $X$, then $X$ is $G$-paradoxical using four pieces. We state and prove some auxiliary results, especially Theorem \ref{zbC} and Lemma \ref{lift}. Finally, we recall the relationship between paradoxical decompositions and amenability of groups.

Section 4 is focused on finding non-Abelian free subgroups in special linear groups $\SL(n,D)$ and special affine groups $\SA(n,D)$  over the {\em ring of integers} $D$ of $K$ (i.e., $D=\{\alpha \in K: \lvert \alpha \rvert\leq 1\}$). Subsection 4.1 deals with the case of $K$ of characteristic zero; its main result is Theorem \ref{Sato-like}. In Subsection 4.2, we investigate the case when $K$ is a transcendental extension of its prime field. We introduce polynomials describing the trace of matrices in $\SL(2,K)$. We prove Theorem \ref{degree_of_psi}, which provides a new tool for finding non-Abelian free subgroups of $\SL(2,D)$ acting without nontrivial fixed points on $K^2\setminus \{\mathbf{0}\}, D^2\setminus \{\mathbf{0}\}$, and any sphere $S[\mathbf{0},r]$, $r>0$, (Theorem \ref{pos_char}) and non-Abelian free subgroups of $\SA(n,D)$ acting without nontrivial fixed points on balls and spheres in $K^n$ (Theorem \ref{affine_general}).

Section 5 contains our main results. Assume that $(K, \lvert\blank\rvert)$ is a non-Archimedean nontrivially valued field and $n\geq 2$. We prove that all balls and spheres in $K^n$ not containing $\mathbf{0}$ are paradoxical with respect to certain groups of linear isometries of $K^n$ using four pieces for balls, and $5-(-1)^n$ pieces for spheres (Theorems \ref{balls_without_0} and \ref{spheres}), $K^n\setminus \{\mathbf{0}\}$ and $D^n\setminus \{\mathbf{0}\}$ (Corollary \ref{K^n_minus_0}). We also prove that all balls and spheres containing $\mathbf{0}$ are paradoxical using four pieces with respect to certain groups of affine isometries of $K^n$ (Theorem \ref{sets_with_0}). Next, we prove a strong result analogical to the classical Banach--Tarski paradox:

{\it If $(K,\lvert\blank\rvert)$ is a non-Archimedean discretely valued field with finite residue field (in particular, if $(K,\lvert\blank\rvert)$ is a non-Archimedean locally compact field) and $n\geq 2$, then any two bounded subsets of $K^n$ with nonempty interiors are $G$-equidecomposable (and $G$-paradoxical), where $G$ is the group of special affine isometries of $K^n$} (see Theorem \ref{strong_paradox_discrete_finite_residue}).

Finally, we prove that the whole space $K^n$ is paradoxical using five pieces with respect to a certain group of affine isometries (Theorem \ref{whole_space}).

\section{Non-Archimedean valued fields and normed spaces over them}

\subsection{Preliminaries}

Throughout this paper, by a {\em field} we mean a commutative division ring.

A {\em valuation} on a field $K$ is any function $\lvert\blank\rvert\colon K\to[0,\infty)$ satisfying for all $x,y\in K$ the following conditions:

(1) $\lvert x \rvert=0 \Leftrightarrow x=0$, (2) $\lvert xy \rvert=\lvert x \rvert\lvert y \rvert$, (3) $\lvert x+y \rvert\leq \lvert x \rvert+\lvert y \rvert$.

The pair $(K,\lvert\blank\rvert)$ is then called a {\em valued field}.
If we replace the condition (3) with a stronger one (3')
$\lvert x+y \rvert\leq \max\{\lvert x \rvert,\lvert y \rvert\},$ the valuation $\lvert\blank\rvert$ will be called {\em non-Archimedean}.

If $(X,d)$ is a metric space, $x_0\in X$ and $r>0$, we put\\
$B[x_0,r]:=\{x\in X: d(x,x_0)\leq r\}$, $B(x_0,r):=\{x\in X: d(x,x_0)<r\}$ and $S[x_0,r]:=\{x\in X: d(x,x_0)= r\}.$

A metric $d$ on a set $X$ is called an {\em ultrametric} if it satisfies the {\em strong triangle inequality}:
$d(x,z) \leq \max\{d(x,y), d(y,z)\}$ for all $x,y,z \in X$. The pair $(X,d)$ is then called an {\em ultrametric space}. It is noteworthy that, in an ultrametric space, any point belonging to a ball (open or closed) of a given radius may play the role of its center.

Any ultrametric space $(X,d)$ has the {\em isosceles property}: If $x,y,z \in X$ and $d(x,y)\neq d(y,z)$, then
    $d(x,z) = \max\{d(x,y), d(y,z)\}$.

Each valuation $\lvert\blank\rvert$ on a field $K$ induces a metric $d\colon K\times K \to [0, +\infty)$, $d(x,y):=\lvert x-y \rvert$. If $\lvert\blank\rvert$ is non-Archimedean, $d$ is an ultrametric. If $d_1$, $d_2$ are metrics on $K$ induced by valuations $\lvert\blank\rvert_1$, $\lvert\blank\rvert_2$ respectively, they induce the same topology on $K$ if and only if $\lvert\blank\rvert_1$ and $\lvert\blank\rvert_2$ are equivalent, i.e., $\lvert\blank\rvert_2=(\lvert\blank\rvert_1)^c$ for some $c>0$ \cite[Theorem 1.2.2]{PG}.

Let $(K,\lvert\blank\rvert)$ be a fixed non-Archimedean valued field. We introduce the following notation:
$D= B[0,1], \mathfrak{m}= B(0,1)$ and $D^* = S[0,1].$
The subset $D$ is a subring of $K$ called the {\em ring of integers} of $K$, $\mathfrak{m}$ is its unique maximal ideal and $D^*$ is the group of units (i.e., invertible elements) of $D$ (cf. \cite[Lemma 1.2]{Schneider}). The quotient ring $k:=D/\mathfrak{m}$ is called the {\em residue field} of $K$.

We denote by $K^{*}$ the multiplicative group of $K$. The set $\lvert K^{*} \rvert:=\{\lvert x \rvert\colon x\in K^{*}\}$ is a subgroup of the multiplicative group $\mathbb{R}_{+}$.
If $\lvert K^{*} \rvert=\{1\}$, the valuation $\lvert\blank\rvert$ is called {\em trivial}. If there exists $t:=\max \lvert K^{*} \rvert\cap(0,1)$, then $\lvert K^{*} \rvert=\{t^n \colon n\in \mathbb{Z}\}$ and the valuation is called {\em discrete}. In the remaining case $\lvert K^{*} \rvert$ is a dense subset of $[0,\infty)$ and the valuation is called {\em dense}.

Assume that $K$ is equipped with a discrete valuation. Fix an arbitrary $\pi\in K$ satisfying $\lvert \pi \rvert=t:=\max \lvert K^{*} \rvert\cap(0,1)$. Such an element $\pi$ will be called a {\em uniformizer}. Then $\lvert a \rvert\leq t^{n}$ holds if and only if $a\in \pi^n D$. Therefore we have
$B[0,t^{n}]=B(0,t^{n-1})=\pi^n D$ for all $n\in\mathbb{Z}$. Each nonzero ideal in $D$ is principal and of the form $\pi^n D$ for some $n\geq 0$, in particular $\mathfrak{m}=\pi D$.

If $(K,\lvert\blank\rvert)$ is a complete non-Archimedean valued field and $L$ is an algebraic extension of $K$, then there is a unique valuation on $L$ that extends $\lvert\blank\rvert$ {\cite[Theorem 14.2]{UC}}. It follows that if $(K,\lvert\blank\rvert)$ is a trivially valued field (e.g., a finite one), the only possible valuation on an algebraic extension $L$ of $K$ that extends $\lvert\blank\rvert$ is trivial.

A non-Archimedean nontrivially valued field $(K, \lvert\blank\rvert)$ is locally compact if and only if it is complete, the valuation $\lvert\blank\rvert$ is discrete, and the residue field $k=D/\mathfrak{m}$ is finite {\cite [Corollary 12.2]{UC}}. A locally compact nontrivially valued field $(K,\lvert\blank\rvert)$ is called a {\em local field}.

Every nontrivial non-Archimedean valuation on $\mathbb{Q}$ is equivalent to the {\em $p$-adic valuation $\lvert\blank\rvert_p$}, for a prime number $p$, defined as follows
 \begin{equation}\label{p_adic_val}
    \lvert 0 \rvert_p:= 0, \; \left\lvert p^k \frac{m}{n}\right \rvert_p:=p^{-k} \quad \text{for } k,m,n\in\mathbb{Z}, \quad p\nmid m, p\nmid n.
\end{equation}
The completion of $(\mathbb{Q}, \lvert\blank\rvert_p)$ is called the {\em field of $p$-adic numbers} \cite[Definition 1.2.7]{PG} and denoted by $(\mathbb{Q}_p, \lvert\blank\rvert_p)$. It is a standard example of a non-Archimedean local field.
The valued field obtained by uniquely extending the valuation $\lvert\blank\rvert_p$ from $\mathbb{Q}_p$ to its algebraic closure and then taking the metric completion is called the {\em field of p-adic complex numbers} and denoted by $(\mathbb{C}_p, \lvert\blank\rvert_p)$ \cite[Example 1.2.11]{PG}. This field is algebraically closed (unlike $\mathbb{Q}_p$) but not locally compact \cite[Corollary 17.2]{UC}. Remarkably, $\mathbb{C}_p$ and $\mathbb{C}$ are isomorphic as fields \cite[p. 83]{vR}.

\subsection{Groups of linear and affine isometries of normed spaces $(K^n, \lVert \blank \rVert )$}

Let $R$ be a commutative ring with the identity element $1$. We write $R^{*}$ for the group of units (i.e., invertible elements) of $R$. We denote by $M_n(R)$ the ring of square matrices of order $n$ over $R$, and by $I_n$ the identity matrix of order $n$.
The group $\GL(n,R)=\left(M_{n}(R)\right)^{*}$ of invertible elements of $M_{n}(R)$ is called the ($n$-th) {\em general linear group} over $R$.

It is well known that $A\in\GL(n,R)$ if and only if the determinant of $A$ is invertible in $R$, i.e.,
\[\GL(n,R)=\{A\in M_{n}(R) \colon \det(A)\in R^* \}.\]
The kernel of the homomorphism $\det \colon \GL(n,R)\to R^*, A \mapsto \det (A),$ is called the ($n$-th) {\em special linear group} over $R$ and is denoted by $\SL(n,R)$, i.e.,
\[\SL(n,R)=\{A\in \GL(n,R) \colon \det(A)=1\}=\{A\in M_{n}(R) \colon \det(A)=1\}.\]

We adopt the convention identifying an $R$-linear mapping $U\colon R^n \to R^n$ with its matrix in the canonical basis $\{e_1,\dots,e_n\}$ of the free $R$-module $R^n$. We will treat vectors in $R^n$ as columns on which matrices operate from the left.

Recall also that an {\em affine mapping} $F\colon R^n \to R^n$ is a mapping of the form
\begin{equation}\label{affine}
   F(x)=L(x)+\tau, \quad x\in R^n,
\end{equation}
where $L\colon R^n \to R^n$ is linear and $\tau \in R^n$ is a fixed vector; $L$ and $\tau$ are called the {\em linear} and {\em translational} part of $F$, respectively. Notice that $F$ is invertible if and only if so is $L$. The group of all invertible affine transformations of $R^n$ is called the ($n$-th) {\em general affine group} over $R$ and denoted by $\GA(n, R)$. The subgroup of $\GA(n, R)$ consisting of all transformations with the linear part belonging to $\SL(n, R)$ is called the ($n$-th) {\em special affine group} over $R$ and is denoted by $\SA(n, R)$.

If $1\le n<m$, we can treat $R^m$ as the product $R^m=R^n \times R^{m-n}$. We then define the {\em canonical embedding} $\iota_{n,m}\colon \GA(n,R)\to \GA(m,R)$ by the formula
\begin{equation}\label{canonical_embedding}
    \iota_{n,m}(F)(x,y):=(F(x),y) \quad \textrm{for } (x,y)\in R^n \times R^{m-n}, \,F\in \GA(n,R).
\end{equation}
This embedding is obviously an injective group homomorphism.

If $F$ has the form \eqref{affine}, then $\iota_{m,n}(F)$ has the linear part equal to $L\times I_{m-n}$ and the translational part equal to $(\tau,0_{m-n})$.

 Let $X$ be a linear space over a non-Archimedean valued field $(K,\lvert\blank\rvert)$. A function $\lVert \blank \rVert \colon X\to [0, \infty)$ is called a {\em non-Archimedean norm} (briefly: norm) on $X$ if for all $x,y \in X$ and $\alpha\in K$ we have the following conditions:

    (1) $ \lVert x \rVert =0 \Leftrightarrow x=0$, (2) $\lVert \alpha x \rVert = \lvert \alpha \rvert \lVert x \rVert $, (3) $\lVert x+y \rVert  \leq \max\{\lVert x \rVert ,\lVert y \rVert \}$.

 Then the pair $(X,\lVert \blank \rVert )$ is called a {\em non-Archimedean normed space} over the non-Archimedean valued field $(K,\lvert\blank\rvert)$.
 The norm $\lVert \blank \rVert $ induces an ultrametric $d$ on $X$, namely $d(x,y):=\lVert x-y \rVert $.

For any $n\in \mathbb N$, the function 
\[\lVert \blank \rVert \colon K^n \to [0, \infty),\quad 
\lVert x \rVert :=\max \{\lvert x_1 \rvert, \ldots, \lvert x_n \rvert\},\]
is a non-Archimedean norm on $K^n$; it is called the {\em standard} norm. Further in the paper, we will denote by $\lVert\blank\rVert$ the standard norm on $K^n$, other norms will not be considered.

\medskip

Let us state a useful lemma saying that if $(K,\lvert\blank\rvert)$ is a non-Archimedean discretely valued field with finite residue field, then, for any set $A$ contained in a ball $B$ in $(K^n, \lVert \blank \rVert)$ with $\operatorname{Int}(A)\ne \emptyset$, the ball $B$ can be covered by a finite number of translates of $A$. This result will be used in the proof of Theorem \ref{strong_paradox_discrete_finite_residue}.

\begin{lem}\label{discrete_finite_residue}
    Let $(K, \lvert\blank\rvert)$ be a non-Archimedean discretely valued field with finite residue field $k$.
    If $r>0$, $n\in \mathbb{N}$, $A\subseteq B[\mathbf{0},r]\subseteq K^n$ and $A$ has nonempty interior, then there exist $m\in \mathbb{N}$ and $a_1,\dots,a_m \in B[\mathbf{0},r]$ such that $B[\mathbf{0},r]\subseteq (a_1 + A) \cup \dots \cup (a_m + A)$.
\end{lem}

\begin{proof}
    Let us fix a uniformizer $\pi\in K$. We then have $B[\mathbf{0},r]=B[\mathbf{0},\lvert \pi \rvert^i]$ for some $i\in \mathbb{Z}$. Since $\operatorname{Int}(A)\ne \emptyset$, there exist $a\in A$ and $j\ge i$ such that $a+B[\mathbf{0},\lvert \pi \rvert^j]\subseteq A$. It will be sufficient to show that the subgroup $B[\mathbf{0},\lvert \pi \rvert^j]$ of $B[\mathbf{0},\lvert \pi \rvert^i]$ has finite index.

Let $q$ be the cardinality of $k$. Notice that $\pi^j D$ has index $q^{j-i}$ in $\pi^i D$. Indeed, if $T$ is a set of coset representatives in $D$ modulo $\pi D$, then the set of all sums $\sum_{s=0}^{j-i-1} \pi^{i+s} t_s$, where $t_0,\dots,t_{j-i-1}\in T,$ is a set of coset representatives in $\pi^i D$ modulo $\pi^j D$.
    Since any $x=(x_1,\dots,x_n)$ and $y=(y_1,\dots,y_n)$ in $B[\mathbf{0},\lvert \pi \rvert^i]$ are congruent modulo $B[\mathbf{0},\lvert \pi \rvert^j]$ if and only if $x_s-y_s \in \pi^j D$ for all $1\le s\le n$, it follows that the index of $B[\mathbf{0},\lvert \pi \rvert^j]$ in $B[\mathbf{0},\lvert \pi \rvert^i]$ equals $q^{n(j-i)}$.
\end{proof}

The following theorem gives a characterization of linear isometries of the normed space $(K^n,\lVert \blank \rVert )$ for a non-Archimedean valued field $(K,\lvert\blank\rvert)$. A similar result with an additional assumption of completeness of the field $(K,\lvert\blank\rvert)$ was proved in \cite[Corollary 1.7]{vRNotes}. Our proof is simpler and direct.

\begin{tw}\label{isoweight}
    Let $(K,\lvert\blank\rvert)$ be a non-Archimedean valued field and $n\in \mathbb N$. $\\$ Let $U: K^n \to K^n$ be a linear map with the matrix $[u_{i,j}]$ (also denoted by $U$) in the canonical basis of $K^n$.

    Then the following conditions are equivalent:
\begin{enumerate}[nosep, label={\em(\arabic*)}]
\item $U$ is a linear isometry of the normed space $(K^n,\lVert \blank \rVert )$,
i.e., $\lVert Ux \rVert =\lVert x \rVert $ for any $x\in K^n$.
\item $\lvert \det(U) \rvert=1$ and $\max_{1\leq i \leq n}\lvert u_{i,j} \rvert=1$ for any $1\leq j\leq n$.
\item $\lvert \det(U) \rvert=1$ and $\lvert u_{i,j} \rvert\leq 1$ for all $1\leq i,j \leq n$.
\item $U$ is invertible and $\lvert u_{i,j} \rvert, \lvert v_{i,j} \rvert \leq 1$ for all $1\leq i,j \leq n,$
where $[v_{i,j}]$ is the matrix of the linear map $V=U^{-1}$.
\item $U\in \GL(n,D)$.
\end{enumerate}
\end{tw}

\begin{proof}
$(1)\Rightarrow (2)$. Let $1\leq j\leq n.$ Then $\max_{1\leq i \leq n}\lvert u_{i,j} \rvert=1$ since $\lVert Ue_j \rVert =\lVert e_j \rVert .$ Hence,
\[\lvert \det(U) \rvert=\left\lvert \sum_{\sigma \in S_n} \mathrm{sgn} (\sigma )\, u_{1,\sigma(1)}\dots u_{n,\sigma (n)}\right \rvert \leq \max_{\sigma \in S_n} \lvert u_{1,\sigma (1)} \rvert\dots \lvert u_{n,\sigma (n)} \rvert \leq 1.\]
Clearly, $U$ is invertible and $U^{-1}$ is a linear isometry, so $\lvert \det(U^{-1}) \rvert\leq 1.$ Thus $1\geq \lvert \det (U) \rvert=\lvert \det(U^{-1}) \rvert^{-1}\geq 1,$ so $\lvert \det (U) \rvert=1.$

$(2)\Rightarrow (3)$ is obvious.

$(3)\Rightarrow (4)$. $U$ is invertible, since $\det (U)\neq 0$.  By Cramer's formula we have
$\lvert v_{i,j} \rvert=\left\lvert \frac{1}{\det(U)}(-1)^{i+j} \det(U_{j,i})\right \rvert=\lvert \det(U_{j,i}) \rvert$, where $U_{j,i}$ is the matrix obtained from $U$ by deleting the $j$-th row and the $i$-th column. Estimating $\lvert \det(U_{j,i}) \rvert$ as above, we get $\lvert \det(U_{j,i}) \rvert\leq 1,$ so $\lvert v_{i,j} \rvert\leq 1$ for all $1\leq i,j \leq n.$

$(4)\Rightarrow (1)$. For $1\leq j \leq n$ we have 
\[\lVert Ue_j \rVert =\left\lVert \sum_{i=1}^n u_{i,j}e_i \right\rVert =\max_{1\leq i \leq n} \lvert u_{i,j} \rvert\leq 1,\]
and for any $x=(x_1, \ldots, x_n) \in K^n$ we have 
\[\lVert Ux \rVert =\left\lVert \sum_{j=1}^n x_j Ue_j \right\rVert \leq \max_{1\leq j \leq n} \lvert x_j \rvert\lVert Ue_j \rVert \leq \max_{1\leq j \leq n} \lvert x_j \rvert= \lVert x \rVert .\] Similarly, we get $\lVert Ve_j \rVert \leq 1$ for $1\leq j \leq n$,  and  $\lVert Vx \rVert \leq \lVert x\lvert $ for any $x\in K^n.$  Hence, we get $ \rVert x\lVert = \rVert VUx\lVert \leq  \rVert Ux\lVert \leq  \rVert x\lVert ,$ so $ \rVert Ux\lVert = \rVert x\ \rvert$ for any $x\in K^n.$

$(4)\Leftrightarrow (5)$ is obvious.
\end{proof}

\begin{uw}\label{glsl}
     Thus, we can treat $\GL(n,D)$ as the group of all linear isometries of the normed space $(K^n,\lVert \blank \rVert )$. The isometries represented by elements of its subgroup $\SL(n,D)$ will be called {\em special linear isometries}.
\end{uw}

There is an analogy with the classical case of linear isometries of the Euclidean space $\mathbb{R}^3$. All such isometries form the orthogonal group
\[\mathrm{O}(3):=\{U\in M_3(\mathbb{R}) \colon U^{T}U=I_3\},\]
whose elements necessarily satisfy $\lvert \det{U} \rvert=1.$
The special orthogonal group
\[\mathrm{SO}(3):=\{U\in \mathrm{O}(3)\colon \det(U)=1\}\]
corresponds to the group of orientation-preserving linear isometries of $\mathbb{R}^3$, which are precisely rotations around an axis containing the origin.

An affine mapping $F\colon K^n \to K^n$ of the form \eqref{affine} is an isometry of the normed space $(K^n,\lVert \blank \rVert )$ if and only if so is its linear part $L$. By $\GA(n,D,K)$ we denote the group of all affine isometries of $(K^n,\lVert \blank \rVert )$, and by $\SA(n,D,K)$ we denote the intersection $\GA(n,D,K) \cap \SA(n, K)$, whose elements will be called {\em special affine isometries}.
Thus,
\[ \SA(n,D,K) = \{x\mapsto L(x) +\tau \colon L\in \SL(n, D), \tau\in K^n\}.\]
Notice also that $\SA(n, D)$ can be regarded as the subgroup of $\SA(n,D,K)$ consisting of all elements with translational parts $\tau \in D^n$.

We will also need two related families of subgroups of $\SL(n,D)$ and $\SA(n,D)$, respectively. Namely, for $\varepsilon\in (0,1]$ and a subring $R\subseteq D$ containing $1$, let us define
\begin{gather*}
    \SL(n,R,\varepsilon):=\{[a_{ij}]\in \SL(n,R)\colon \lvert a_{ii}-1 \rvert<\varepsilon, \lvert a_{ij} \rvert<\varepsilon \; \textrm{for } i\ne j, 1\le i,j\le n\},\\
    \SA(n,D,\varepsilon):=\{x\mapsto L(x) +\tau \colon   L\in \SL(n, D,\varepsilon), \tau\in D^n \}.
\end{gather*}

\section{Paradoxical sets with respect to group actions}

Most of the following definitions and facts are taken from the monograph \cite{W} and generalized in some cases. All group actions are assumed to be left. A set $X$ together with a specified action of a group $G$ on it will be called {\em a $G$-set}.
For the union of (a family of) sets, the symbol ``$\sqcup$'' (or ``$\bigsqcup$'') will indicate disjointness of the components.

    We say that a nonempty subset $E$ of a $G$-set $X$ is {\em $G$-paradoxical} \cite[Def. 1.1]{W}, if for some positive integers $m$, $n$ there exist pairwise disjoint subsets $A_1,\dots,A_m,B_1,\dots,B_n$ of $E$ and elements $g_1,\dots,g_m,h_1,\dots,h_n\in G$ such that
    \begin{equation}\label{slpar}
        E=\bigcup_{i=1}^{m} g_i(A_i) = \bigcup_{j=1}^{n} h_j(B_j).
    \end{equation}

    Then there exist positive integers $m$, $n$, subsets $A_1,\dots,A_m,B_1,\dots,B_n$ of $E$ and elements $g_1,\dots,g_m,h_1,\dots,h_n\in G$ \cite[Lemma A.12]{KateJ} such that
    \begin{equation}\label{rparwzor}
        E=A_1 \sqcup \dots \sqcup A_m \sqcup B_1 \sqcup \dots \sqcup B_n =\bigsqcup_{i=1}^{m} g_i(A_i) = \bigsqcup_{j=1}^{n} h_j(B_j).
    \end{equation}
    In such a case the decomposition \eqref{rparwzor} of $E$ is called {\em a $G$-paradoxical decomposition}.

    Additionally, if the subset $E$ of $X$ is $G$-invariant, we can assume that $g_1=h_1=1_{G}$ and the $g_i$'s, $h_j$'s, $m$, $n$ from \eqref{rparwzor} are the same as those from \eqref{slpar} (cf. \cite[Lemma A.12]{KateJ}).

    We say that subsets $A$ and $B$ of a $G$-set $X$ are {\em $G$-equidecomposable} \cite[Def. 3.4]{W} if for some $n\in\mathbb{N}$ there exist decompositions
    \begin{equation}\label{GEq}
        A=\bigsqcup_{i=1}^n A_i, \; B=\bigsqcup_{i=1}^n B_i,
    \end{equation}
    and elements $g_1,\dots,g_n\in G$ such that $g_i(A_i)=B_i$ for any $1\leq i\leq n.$ Then we write $A\sim_{G} B$ or $A\sim B$. If we want to emphasize that the decompositions \eqref{GEq} can be realized with $n$ pieces, we write $A\sim_{G, n} B$ or $A\sim_{n} B$.

     Clearly, a subset $E$ of a $G$-set $X$ is $G$-paradoxical if and only if $E$ is the union of disjoint subsets $A$ and $B$ such that $A\sim_m E \sim_n B$ for some $m,n\geq 1$. We then say that $E$ has a $G$-paradoxical decomposition {\em using $r=m+n$ pieces}. Obviously, if $E \subseteq X$ is $H$-paradoxical (using $r$ pieces) for some $H\le G$, then $E$ is $G$-paradoxical (using $r$ pieces).

    Let $X$ be a $G$-set and $H\le G$. It is not hard to check that if $E\subseteq X$ is $H$-paradoxical using $r$ pieces, then $g(E)$ is $H$-paradoxical using $r$ pieces for any $g\in G$ with $gHg^{-1}\le H$.

If not stated otherwise, in the case $X=G$ we mean the action of $G$ on itself by left multiplication. If a group $G$ is $G$-paradoxical, the least number $r$ for which there exists a paradoxical decomposition of $G$ using $r$ pieces is called the {\em Tarski number} of $G$ \cite{S}. In the same way we talk about the Tarski number of a $G$-paradoxical subset $E$ of a $G$-set $X$; clearly it is greater or equal to four.

Since all free groups of a given rank are isomorphic, we will sometimes abuse notation and write $F_n$ for any free group of rank $n$. By {\em the} free group $F_n$ we will denote the model group of reduced words over an alphabet consisting of $n$ symbols and their formal inverses.

The free group $F_2$ is the most important example of a paradoxical group. Its Tarski number is four \cite[Theorem 5.2]{W}. Indeed, let $\{\tau, \sigma\}$ be a basis of $F_2$. We will denote by $W(\xi)$ the set of all elements (reduced words) of $F_2$ which begin with the letter $\xi\in\{\tau,\tau^{-1},\sigma,\sigma^{-1}\}$, and $1_{F_2}$ will stand for the empty word. Let $A_1:= W(\tau), A_2:= W(\tau^{-1}),$
$B_1:= W(\sigma) \cup \{1_{F_2}\} \cup \{\sigma^{-n}\colon n\ge 1\}$ and $B_2:=  W(\sigma^{-1})\setminus \{\sigma^{-n}\colon n\ge 1\}$. It is easy to check that
    $F_2=A_1 \sqcup A_2 \sqcup B_1 \sqcup B_2=A_1 \sqcup \tau(A_2)=B_1 \sqcup \sigma(B_2);$
    hence, $A\sim_2 F_2 \sim_2 B,$ where $A:= A_1 \sqcup A_2$ and $B:= B_1 \sqcup B_2$.

If $G$ acts on $X$ and $x\in X$, the subgroup $\mathrm{Stab}_G(x):=\{g\in G\colon g(x)=x\}$ of $G$ is called the {\em stabilizer} of $x$. When $\mathrm{Stab}_G(x)=\{1_{G}\}$ for all $x\in X$, then we say that $G$ acts on $X$ {\em without nontrivial fixed points}. If $\mathrm{Stab}_G(x)$ is an Abelian group for every $x\in X$, we say that $G$ acts on $X$ {\em locally commutatively}.

    If a group $G$ is $G$-paradoxical using $r$ pieces and $G$ acts on $X$ without nontrivial fixed points, then the $G$-set $X$ is $G$-paradoxical using $r$ pieces \cite[Proposition 1.10]{W}. Hence, if a group $G$ acts on $X$ without nontrivial fixed points and $G$ contains a free subgroup $F_2$, then the set $X$ is $G$-paradoxical using $4$ pieces.

 Under the axiom of choice, we have the following characterization of $G$-paradoxicality using $4$ pieces.
A $G$-set $X$ is $G$-paradoxical using $4$ pieces if and only if $G$ contains a free subgroup $F_2$ whose action on $X$ is locally commutative. More precisely, $X$ has a decomposition
    \begin{equation}\label{characterization_four_pieces}
        X=A_1 \sqcup A_2 \sqcup A_3 \sqcup A_4 = A_1 \sqcup \sigma(A_2) = A_3 \sqcup \tau(A_4)
    \end{equation}
    for $\sigma,\tau \in G$ if and only if $\{\sigma,\tau\}$ is a basis of a free subgroup $F_2\le G$ whose action on $X$ is locally commutative \cite[Theorems 5.5 and 5.8]{W}.

\begin{tw}\label{slabyrozkl}
    Let $X$ be a $G$-set, $A,B,C,Z\subseteq X$ be pairwise disjoint and $E=A \sqcup B \sqcup C$. Let $H$ be a subgroup of $G$. If $A\sim_{H,m} E \sim_{H,n} B$ and $C\sim_{G,l} Z$, the set $E \sqcup Z$ is $G$-paradoxical using $m+n+l+1$ pieces.

    If, additionally, $A$, $B$ or $E$ is $H$-invariant, then $E\sqcup Z$ is $G$-paradoxical using $m+n+l$ pieces.
\end{tw}

\begin{proof}
    From the assumption we have
    \begin{equation}\label{4zbiory}
        E\sqcup Z = (A\sqcup Z) \sqcup (B\sqcup C).
    \end{equation}
    From the relations $B\sim_{H,n} E$ and $C\sim_{G,l} Z$, having in mind the disjointness of $B$ with $C$ and $E$ with $Z$, we conclude that $(B\sqcup C)\sim_{G,n+l} (E\sqcup Z)$.
    Similarly we justify $(A\sqcup Z)\sim_{H,m+1} (E\sqcup Z)$, which finishes the proof of the first part of our theorem.

    If $A$ or $E$ is $H$-invariant, the relation $A\sim_{H,m} E$ can be realized by decompositions
    \[A=A_1 \sqcup \left( \bigsqcup_{i=2}^{m} A_i \right), \quad E= A_1 \sqcup \left( \bigsqcup_{i=2}^{m} \alpha_i (A_i)\right),\]
     for some $\alpha_2, \ldots, \alpha_m \in H$. That means, $A_1$ and $Z$ can be joined together into $A_1 \sqcup Z$ so that we obtain the relation $(A\sqcup Z)\sim_{H,m} (E\sqcup Z)$ and diminish by one the number of pieces sufficient to a $G$-paradoxical decomposition of $E\sqcup Z$.
    If $B$ is $H$-invariant, one needs to repeat the reasoning, interchanging the roles of $A$ and $B$ in \eqref{4zbiory}.
\end{proof}

An elementary example of the situation $A\sim_G E \sim_G B$, $E=A\sqcup B \sqcup C$ with a nonempty $C$ is the action of the free group $F_2$ over the alphabet $\{\sigma,\tau\,\sigma^{-1},\tau^{-1}\}$ on itself. Let $A_1=W(\sigma)$, $A_2=W(\sigma^{-1})$, $B_1=W(\tau)$, $B_2=W(\tau^{-1})$, $C=\{1_{F_2}\}$ and $A=A_1 \sqcup A_2$, $B=B_1\sqcup B_2$, where $W(\xi)$ denotes the set of all elements of $F_2$ which begin with a letter $\xi \in \{\sigma,\tau\,\sigma^{-1},\tau^{-1}\}$. It is easy to check that
\[F_2=A_1 \sqcup \sigma(A_2)=B_1 \sqcup \tau(B_2),\]
so $A\sim_2 F_2 \sim_2 B$. We will make us of this in a more general situation of a locally commutative action of $F_2$ on $X$.

\begin{tw}\label{zbC}
Let $X$ be a $G$-set. Assume that $G$ contains a free subgroup $F=F_2$ which acts locally commutatively on a subset $E$ of $X$. Let $C$ be a subset of $E$ such that  $\mathrm{Stab}_F(x)=\{1_F\}$ and $Fx\cap C=\{x\}$ for every $x\in C.$ Then $E$ has a decomposition $E=A_1\sqcup A_2\sqcup B_1\sqcup B_2\sqcup C$ such that $E=A_1 \sqcup \sigma (A_2)= B_1 \sqcup \tau (B_2),$ where $\{\sigma, \tau\}$ is a basis of $F$.

If $Z\subseteq X\setminus E$ and $Z\sim_{G,l}C,$ then $Z \sqcup E$ is $G$-paradoxical using $4+l$ pieces. If $\alpha (C)\subseteq X\setminus E$ for some $\alpha \in G$ and $Y:=\alpha (C) \sqcup E,$ then 
\[Y= (\alpha (C) \sqcup A_1) \sqcup A_2\sqcup B_1\sqcup B_2\sqcup C=(\alpha (C) \sqcup A_1) \sqcup \sigma (A_2)=B_1 \sqcup \tau (B_2) \sqcup \alpha (C),\]
so $\alpha (C) \sqcup E$ is $H$-paradoxical using $5$ pieces, where $H$ is the subgroup of $G$ generated by $\sigma, \tau$ and $\alpha$.
    \end{tw}

\begin{proof}
For any disjoint subsets $U,V$ of $F$ and $x\in C$ we have $Vx \cap Ux=\emptyset,$ since $\mathrm{Stab}_F(x)=\{1_F\}.$ Moreover, $Fx\cap Fy=\emptyset$ for any different $x,y\in C$, since $Fx\cap C=\{x\}.$ The group $F$ acts locally commutatively on $E$ and the set $E_1:=\bigsqcup_{x\in C} Fx$ is $F$-invariant, so $F=F_2$ acts locally commutatively on the set $E_2:= E\setminus E_1.$ Thus, $E_2$ has a decomposition \[E_2=A^2_1\sqcup A^2_2 \sqcup B^2_1 \sqcup B^2_2\] such that  \[E_2=A^2_1\sqcup \sigma (A^2_2) = B^2_1 \sqcup \tau (B^2_2). \]   The sets $A^1_1:=\bigsqcup_{x\in C} W(\sigma) x, A^1_2:=\bigsqcup_{x\in C} W(\sigma^{-1}) x, B^1_1:=\bigsqcup_{x\in C} W(\tau) x$ and $B^1_2:=\bigsqcup_{x\in C} W(\tau^{-1}) x$ are pairwise disjoint and \[E_1=A^1_1\sqcup A^1_2 \sqcup B^1_1 \sqcup B^1_2 \sqcup C= A^1_1\sqcup \sigma (A^1_2) = B^1_1 \sqcup \tau (B^1_2). \]
It follows that \[E=A_1\sqcup A_2 \sqcup B_1 \sqcup B_2 \sqcup C= A_1\sqcup \sigma (A_2) = B_1 \sqcup \tau (B_2),\] where $A_1=A^1_1\sqcup A^2_1, A_2=A^1_2\sqcup A^2_2, B_1=B^1_1\sqcup B^2_1$ and $B_2=B^1_2\sqcup B^2_2.$

Let  $Z\subseteq X\setminus E$ with $Z\sim_{G,l}C.$ Then $C$ has a decomposition $C=\bigsqcup_{i=1}^l C_i$ such that $Z=\bigsqcup_{i=1}^l \alpha_i (C_i)$ for some $\alpha_1, \ldots, \alpha_l \in G$. Thus,
\[Z\sqcup E=(Z\sqcup A_1)\sqcup A_2 \sqcup B_1 \sqcup B_2 \sqcup \bigsqcup_{i=1}^l C_i=(Z\sqcup A_1)\sqcup \sigma (A_2)=B_1 \sqcup \tau (B_2) \sqcup \bigsqcup_{i=1}^l \alpha_i (C_i),\]
so $Z\sqcup E$ is $G$-paradoxical using $4+l$ pieces.

Hence, for $Z:=\alpha (C)$ we get the last part of the theorem. \end{proof}

In particular, we have obtained a tool for proving paradoxicality of some sets using $5$ pieces. It is a general statement of the method used originally by Robinson \cite{Rob} (see also\cite[Theorem 5.7]{W}) for getting a paradoxical decomposition of a ball in $\mathbb{R}^3$ using $5$ pieces, based on an action of a free subgroup of the group of rotations $\mathrm{SO}(3)$.

Now we will prove  a simple but important lemma, which will be used several times in our paper. It will serve as a main tool to ``lift'' some paradoxical decompositions of subsets of $K^2$ to higher dimensions.

\begin{lem}\label{lift}
    Let $X$ be a $G$-set, $X'$ a $G'$-set and {$\iota\colon G \to G'$} a group homomorphism. Assume that $E\subseteq X$, $E'\subseteq X'$, $E'$ is {$\iota(G)$-invariant} and a function $f\colon E' \to E$ satisfies
    \begin{equation}\label{gof}
        g\circ f = f \circ (\iota(g)) \quad \textrm{for all } \; g\in G.
    \end{equation}
    If $A\subseteq E$ and $A\sim_G E$ using $m$ pieces, then $f^{-1}(A) \sim_{\iota(G)} E'$ using $m$ pieces.
    In particular, if $E$ is $G$-paradoxical using $r$ pieces, then $E'$ is $\iota(G)$-paradoxical using $r$ pieces.
\end{lem}

\begin{proof}
    From the assumption we have
    \[A=\bigsqcup_{i=1}^m A_i, \quad E=\bigsqcup_{i=1}^m g_i (A_i), \quad g_1,\dots,g_m \in G;\]
    therefore also
    \[f^{-1}(A)=\bigsqcup_{i=1}^m f^{-1}(A_i), \quad E'=\bigsqcup_{i=1}^m f^{-1}(g_i (A_i)), \quad g_1,\dots,g_m \in G.\]
    It follows from \eqref{gof} that, for any $i\in\{1,\dots,m\}$, we have
    \begin{align*}
        f^{-1} (g_i(A_i)) & = \{x\in E' \colon f(x)\in g_i (A_i)\}=\{x\in E' \colon (g_i^{-1} \circ f) (x) \in A_i \}                         \\
                          & =\{x\in E'\colon (f \circ \iota(g_i^{-1}) )(x)\in A_i \} = \{x\in E' \colon \iota(g_i^{-1}) (x) \in f^{-1}(A_i)\} \\
                          & = \iota(g_i) (f^{-1}(A_i)),
    \end{align*}
    so $E'=\bigsqcup\limits_{i=1}^m \iota(g_i) (f^{-1}(A_i))$, which shows that $f^{-1}(A) \sim_{\iota(G),m} E'$.
\end{proof}

If a group $G$ acts on $X$ and $E\subseteq X$ is $G$-paradoxical, then $E$ is $G$-equidecomposable with any subset $A$ of $E$ such that
    \begin{equation}\label{skonpokr}
        E\subseteq g_1 A \cup \dots \cup g_n A \quad \textrm{for some } n\in\mathbb{N}, \; g_1,\dots,g_n \in G.
    \end{equation}
Hence, any two such subsets $A', A''$ of $E$ are $G$-equidecomposable and both $G$-paradoxical \cite[Corollary 10.22]{W}.

\subsection{Amenability and nonexistence of paradoxical decompositions}

The concept of amenability (i.e., admitting some invariant measure by a group $G$) is closely connected to nonexistence of paradoxical decompositions of $G$-sets and their subsets.

    Let $X$ be a $G$-set. A finitely additive measure
    $\mu\colon \mathcal{P}(X)\to [0,\infty]$ is called {\em $G$-invariant} if it satisfies
    \begin{equation}\label{inv}
        \mu(g A)=\mu(A) \quad \textrm{for any } \; g \in G, \, A \subseteq X.
    \end{equation}
    If $X=G$, we say that such $\mu$ is {\em left-invariant}.

    A group $G$ is {\em amenable} if there exists a left-invariant, finitely additive measure $\mu\colon \mathcal{P}(G)\to [0,1]$ with $\mu(G)=1$.

If $\mu\colon \mathcal{P}(X)\to [0,\infty]$ is a $G$-invariant, finitely additive measure on a $G$-set $X$ and $E\subseteq X$ is $G$-paradoxical with $\mu(E)<\infty$, then $\mu(E)=2\mu(E)$ and so $\mu(E)=0$. Thus the existence of a $G$-invariant measure on $X$ excludes $G$-paradoxical decomposition of any $E\subseteq X$ with $0<\mu(E)<\infty$.

In fact, the converse also holds (if the axiom of choice is assumed). The following fact is known as Tarski's theorem (see \cite{Tar38} and a modern approach \cite[Corollary 11.2]{W}):
A subset $E$ of a $G$-set $X$ is $G$-paradoxical if and only if there is no $G$-invariant measure $\mu$ on $X$ with $\mu (E)=1$.

It follows that a group $G$ is amenable if and only if $G$ is not $G$-paradoxical \cite[Theorem A.13]{KateJ}.

If an amenable group $G$ acts on a set $X$, then $X$ is not $G$-paradoxical; in fact, any subset of $X$ containing a nonempty $G$-invariant subset is not $G$-paradoxical (see \cite[Theorem 12.3]{W} and its proof). 

Let us recall a few basic facts about amenable groups \cite[Theorem 12.4]{W}, \cite[Theorem 3.2]{KateJ}:
\begin{enumerate}[nosep, label=(\arabic*)]
    \item All finite and Abelian groups are amenable.
    \item A subgroup of an amenable group is amenable.
    \item If $H$ is a normal subgroup of a group $G$, then $G$ is amenable if and only $H$ and $G/H$ are both amenable.
    \item If $G$ is the direct union of a directed system $(G_{\alpha})_{\alpha\in I}$ of amenable groups, then $G$ is amenable.
    \item All solvable groups are amenable.
\end{enumerate}
It follows from (4) that any {\em locally finite} group (i.e., such that its every finite subset generates a finite subgroup) is amenable. 

If a group $G$ contains a free subgroup $F_2$ of rank two, then $G$ is not amenable, since $F_2$ is $F_2$-paradoxical, so nonamenable. We will see that for an important class of groups the converse also holds.

There is a famous work of Tits \cite{Tits} classifying linear groups, i.e., subgroups of the group $\GL(V)$ of all linear automorphisms of a linear space $V$ over a field $K$. Its main results can be reformulated in the context of amenability. The next statement is a conclusion from \cite[Theorem 1]{Wang} and \cite[Theorem 2]{Tits} (see also Theorem 12.6 in \cite{W} and the discussion following it).

\begin{prop}\label{Tits'}
    Let $K$ be a field and $n\in\mathbb{N}$. Let $G$ be a subgroup of $\GL(n,K)$. Then $G$ is amenable or has a free subgroup of rank two.
\end{prop}

\begin{proof}
    Suppose that $G$ has no subgroup of rank two. If $\ch{K}=0$, then G has a normal solvable subgroup of finite index \cite[Theorem 1]{Wang}, so $G$ is amenable. If $\ch{K}>0$, then $G$ has a normal solvable subgroup $N$ such that $G/N$ is locally finite \cite[Theorem 2]{Tits}, so $G$ is amenable.
 \end{proof}

Let us present a nontrivial example of an amenable linear group.

\begin{ex}\label{alg_ext_of_finite}
    Let $K$ be a finite field and $L$ an algebraic extension of $K$. It is known that $L$ can be expressed as the union $L=\bigcup_{\alpha\in I} K_{\alpha}$ of all subfields $K\subseteq K_{\alpha}\subseteq L$ such that $K_{\alpha}$ is a finite extension of $K$ (in particular, also a finite field). The system $(K_{\alpha})_{\alpha\in I}$ is directed by inclusion. Therefore, for any $n\in \mathbb{N}$, there is a decomposition
    \[\GL(n,L)=\bigcup_{\alpha\in I} \GL(n,K_{\alpha})\]
    of $\GL(n,L)$ as the direct union of a directed system of finite groups, so
    $\GL(n,L)$ is amenable.
\end{ex}

The following simple lemma will be used in the proof of Theorem \ref{trivial_valuation}.

\begin{lem}\label{A(n,K)_amenable}
    Let $K$ be a field and $n\in \mathbb N$. If $\GL(n,K)$ is amenable, then so is $\GA(n,K)$.
\end{lem}

\begin{proof}
    Let $\Phi \colon \GA(n,K) \to \GL(n,K)$ be the epimorphism sending an affine transformation to its linear part. Clearly,
    its kernel is the group of all translations $x\mapsto x+\tau$, $\tau\in K^n$. Hence, $\Ker{\Phi}\cong (K^n,+)$ is amenable as an Abelian group. Since $\GA(n,K)/\Ker{\Phi}\cong \GL(n,K)$ is amenable by the assumption, so is $\GA(n,K)$.
\end{proof}

\section{Some free subgroups of special linear groups $\SL(n,D)$ and special affine groups $\SA(n,D)$}

Let us begin with a simple observation from linear algebra.

\begin{lem}\label{trace}
    Let $K$ be a field and $M\in \SL(2,K)$. Then $\tr{M}=2$ if and only if the linear map $M\colon K^2 \to K^2$ has a nonzero fixed point.
\end{lem}

\begin{proof}
    The characteristic polynomial of $M$ has the form
\begin{equation}\label{charpol}
    f(\lambda)=\det(M-\lambda I_2)=\lambda^2 - (\tr M) \lambda + 1,
\end{equation}
    so $\det(M-I_2)=f(1)=0$ if and only if $\tr M=2$.
\end{proof}

The following lemma gives some conditions under which a group of isometries of $K^n$ acts on certain balls and spheres.

\begin{lem}\label{ultrametric_lemma}
    Let $(K,\lvert\blank\rvert)$ be a non-Archimedean valued field and $n\in \mathbb{N}$. Let $G$ be a group of isometries of the normed space $(K^n,\lVert \blank \rVert )$ and $\Sigma\subseteq G$ a generating set of $G$. Let $x\in K^n$ and $r>0$ be fixed. If $\lVert g(x)-x \rVert <r$ for all $g\in \Sigma$, then the balls $B[x,r]$, $B(x,r)$, and the sphere $S[x,r]$ are $G$-invariant.
\end{lem}

\begin{proof}
    Let $g\in \Sigma$. Since $g$ is a surjective isometry, the images of $B[x,r]$, $B(x,r)$, $S[x,r]$ under $g$ are equal to $B[g(x),r]$, $B(g(x),r)$, $S[g(x),r]$, respectively.
    Since $\lVert g(x)-x \rVert <r$, we obtain $B[g(x),r]=B[x,r]$ and $B(g(x),r)=B(x,r)$ (by the strong triangle inequality), as well as $S[g(x),r]=S[x,r]$ (by the isosceles property). Thus, the sets under consideration are invariant with respect to all $g\in \Sigma$, hence also $G$-invariant.
\end{proof}

The proposition below provides a tool for finding explicit bases of some free groups of affine mappings acting on certain balls and spheres in $K^n$ without nontrivial fixed points.

\begin{prop}\label{free_in_affine}
    Let $K$ be a field and $n\in\mathbb{N}$.
    Assume that a family $\{A_i\}_{i\in S}$ is a basis of a free subgroup $F\le \GL(n,K)$ such that $\det(M-I_n)\ne 0$ for all $M\in F\setminus\{I_n\}$. Let $\hat{x} \in K^n$ and
    $u_i:=(I_n-A_i)\,\hat{x}$ for $i\in S$.
    Then, the family $\{T_i\}_{i\in S}$, where
    \begin{equation}\label{T_i}
        T_i\colon K^n\to K^n, \quad T_i(x):=A_i(x)+u_i,
    \end{equation}
    is a basis of a free subgroup $F'$ of $\GA(n,K)$ acting on $K^n\setminus \{\hat{x}\}$ without nontrivial fixed points.

    If, additionally, $K$ is equipped with a non-Archimedean valuation $\lvert\blank\rvert$, $x\in K^n\setminus \{\hat{x}\}$, $0<r<\lVert x-\hat{x} \rVert $ and $F$ is a subgroup of $\SL(n,D,\varepsilon)$ for some $0<\varepsilon\le \varepsilon_0$, where
    \begin{equation}\label{choice_of_epsilon}
        \varepsilon_0:=
    \begin{cases}
       \min\left(\frac{r}{\lVert x-\hat{x} \rVert }, \frac{1}{\lVert \hat{x} \rVert }\right) & \textrm{if } \hat{x}\ne \mathbf{0},\\
       \frac{r}{\lVert x \rVert } & \textrm{if } \hat{x}= \mathbf{0},
    \end{cases}
    \end{equation}
    then $F'\le \SA(n,D,\varepsilon)$
    and $F'$ acts on the balls $B[x,r]$, $B(x,r)$, and the sphere $S[x,r]$ without nontrivial fixed points.
\end{prop}

\begin{proof}
    Denote by $F'$ the subgroup of $\GA(n,K)$ generated by the family $\{T_i\}_{i\in S}$.
    Let $\Phi \colon \GA(n,K) \to \GL(n,K)$ be the homomorphism sending an affine transformation to its linear part.
    Since $\Phi$ maps $\{T_i\}_{i\in S}$ one-to-one onto the basis $\{A_i\}_{i\in S}$ of $F$, the group $F'$ is free with $\{T_i\}_{i\in S}$ as basis \cite[Proposition 1.8]{LS} and $\phi:=\Phi\vert_{F'}\colon F'\to F$ is an isomorphism.
    
    If $T=M+u \in F'$ and $T\ne 1_{F'}$, then $M=\phi(T)\ne I_n$, so $\det(I_n-M)\ne 0$ and $(I_n-M)^{-1}\, u$ is a unique fixed point of $M$ in $K^n$.
    On the other hand, $\hat{x}$ is a fixed point of $T_i$ for all $i\in S$ and, as a consequence, for all $T\in F'$. Therefore, $F'$ acts on $K^n\setminus \{\hat{x}\}$ without nontrivial fixed points.

    Suppose additionally that the assumptions of the second part of the proposition are satisfied. Clearly, the sets $B[x,r]$, $B(x,r)$, $S[x,r]$ are contained in $K^n\setminus\{\hat{x}\}$. Since $F\le \SL(n,D,\varepsilon)$ and $\lVert u_i \rVert \le \varepsilon \lVert \hat{x} \rVert \le \varepsilon_0 \lVert \hat{x} \rVert \le 1$ for all $i\in S$, we have 
    $F'\le \SA(n,D,\varepsilon)$.
    The choice of $\varepsilon_0$ guarantees that
    \begin{equation*}
        \lVert T_i(x)-x \rVert =\lVert (A_i-I_n)\, x+u_i \rVert =\lVert (A_i-I_n)(x-\hat{x}) \rVert 
        <\varepsilon \lVert x-\hat{x} \rVert \le\varepsilon_0 \lVert x-\hat{x} \rVert  \le r
    \end{equation*}
    for all $i\in S$. Therefore, by Lemma \ref{ultrametric_lemma}, $F'$ acts on $B[x,r]$, $B(x,r)$ and $S[x,r]$ without nontrivial fixed points.
\end{proof}

\begin{uw}\label{remark_open_ball}
    To ensure that the group $F'$ from Proposition \ref{free_in_affine} acts without nontrivial fixed points on the open ball $B(x,r)$, a weaker assumption on $r$ is sufficient, namely $0< r\le \lVert x-\hat{x} \rVert $.
\end{uw}

\subsection{The case when the characteristic of $K$ is zero}

Let $K$ be a field of characteristic $0$. We will identify its prime field with the field $\mathbb{Q}$ of rational numbers. If $K$ is equipped with a non-Archimedean valuation and $D$ is the ring of integers of $K$, then $\mathbb{Z}\subseteq D$ and $\SL(2,\mathbb{Z})\le \SL(2,D)$.

 We will rely on the following remarkable result due to Neumann \cite{N} and Magnus \cite{M}.

\begin{tw}[{\cite[Theorems 3 and 4]{M}}]\label{Neumann_Magnus}
    The matrices
    \begin{equation}\label{Magnus_matrices}
        A_0:=\left(
    \begin{array}{cc}
            1 & 1 \\
            1 & 2
        \end{array}
    \right),\quad
    A_n:=\left(
    \begin{array}{cc}
            4n^2 + 1 & 2n \\
            2n & 1
        \end{array}
    \right),\quad
        n=1,2,\dots
    \end{equation}
    form a basis of a free subgroup $F$ of $\SL(2,\mathbb{Z})$. Moreover, every element of $F\setminus\{I_2\}$ is nonparabolic, i.e., its trace is different from $\pm 2$. 
\end{tw}

\begin{uw}
    The first part of the above-stated result corresponds to \cite[Theorem 3]{M}, which is directly based on Neumann's paper \cite{N}. The second part is contained in \cite[Theorem 4]{M} and proved thereafter. 
\end{uw}

 Taking into account Theorem \ref{Neumann_Magnus} and Lemma \ref{trace}, we obtain an immediate conclusion.

\begin{wn}\label{free_in_char_0}
    If $(K,\lvert\blank\rvert)$ is a non-Archimedean valued field and $\ch{K}=0$, the group $\SL(2,\mathbb{Z})\le \SL(2,D)$ contains a free subgroup of infinite rank acting on $K^2\setminus\{\mathbf{0}\}$, and (for any $r>0$) on $B[\mathbf{0},r]\setminus\{\mathbf{0}\}$ and $S[\mathbf{0},r]$ without nontrivial fixed points.
\end{wn}

If $(K,\lvert\blank\rvert)$ is a non-Archimedean valued field of characteristic $0$, we distinguish two cases depending on the characteristic of $k$, the residue field of $K$.
We either have $\ch{k}=0$, in which case the valuation $\lvert\blank\rvert$ restricted to the prime field $\mathbb{Q} \subseteq K$ is trivial, or $\ch{k}=p$ for a prime number $p$. In the latter case, the valuation restricted to $\mathbb{Q}$ is equivalent to the $p$-adic valuation $\lvert\blank\rvert_p$, given by \eqref{p_adic_val}.

The following results provide the existence of some free groups of linear and affine isometries of the normed space $(K^n,\lVert \blank \rVert )$, in the case when $\ch{K}\ne \ch{k}$, acting on certain sets without nontrivial fixed points. They will be used in Section $5$.

\begin{prop}\label{prop_SL(2,Z)}
    Let $(K,\lvert\blank\rvert)$ be a non-Archimedean valued field and $k$ its residue field. Assume that $\ch{K}=0$ and $\ch{k}>0$. Then, for any $\varepsilon\in (0,1]$, there exists an infinite set $S(\varepsilon)\subseteq \mathbb{N}$ such that the family $\{A_i\}_{i\in S(\varepsilon)}$, where $A_i$ are as in \eqref{Magnus_matrices}, is a basis of a free subgroup $F(\varepsilon)\le \SL(2,\mathbb{Z},\varepsilon)$.
\end{prop}

\begin{proof}
    We can see from \eqref{Magnus_matrices} that, for $n\ge 1$, all entries of $(A_n-I_2)$ have valuation bounded from above by $\lvert n \rvert$. Let $p:=\ch{k}$. Since $\lvert p \rvert<1$, we have $A_{p^m}\in \SL(2,\mathbb{Z},\varepsilon)$ for sufficiently large $m\in\mathbb{N}$.
\end{proof}

\begin{tw}\label{Sato-like}
    Let $(K,\lvert\blank\rvert)$ be a non-Archimedean valued field and $k$ its residue field. Assume that $\ch{K}=0$ and $\ch{k}>0$. Let $n\ge 2$, $\hat{x}=(\hat{x}_1,\hat{x}_2,\dots,\hat{x}_n)\in K^n$, $x=(x_1,x_2,\dots,x_n)\in K^n$, $(x_1,x_2)\ne (\hat{x}_1,\hat{x}_2)$, $0<r<\lVert (x_1,x_2)-(\hat{x}_1,\hat{x}_2) \rVert $ and $\varepsilon\in (0,1]$. Then, there exists a free subgroup $F'\le \SA(n,D,\varepsilon)$ of infinite rank acting without nontrivial fixed points on $K^n\setminus \left(\{(\hat{x}_1,\hat{x}_2)\}\times K^{n-2}\right)$, the balls $B[x,r]$, $B(x,r)$, and the sphere $S[x,r]$.

    Moreover, $F'$ can be chosen so that $F'\le \SA(n,\mathbb{Z})$ if $(\hat{x}_1,\hat{x}_2)\in \mathbb{Z}^2$, and $F'\le \SL(n,\mathbb{Z},\varepsilon)$ if $(\hat{x}_1,\hat{x}_2)=(0,0)$.
\end{tw}

\begin{proof}
    First, let us consider the case $n=2$.
    Let $\varepsilon_0$ be defined as in \eqref{choice_of_epsilon} and $\varepsilon_1:=\min(\varepsilon,\varepsilon_0)$.
    Applying Proposition \ref{prop_SL(2,Z)}, we obtain a free subgroup $F(\varepsilon_1)\le \SL(2,\mathbb{Z},\varepsilon_1)$ with $\{A_i\}_{i\in S(\varepsilon_1)}$ as basis. Notice that $\det(M-I_2)=2-\tr M\ne 0$ for all nonidentity $M\in F(\varepsilon_1)$.
    A direct application of Proposition \ref{free_in_affine} yields a desired free subgroup $F'\le \SA(2,D,\varepsilon_1)\le \SA(2,D,\varepsilon)$ with a basis $\{T_i\}_{i\in S(\varepsilon_1)}$ as in \eqref{T_i}.

    Assume now that $n>2$. Let us again apply Proposition \ref{free_in_affine}, but with $2$ in place of $n$, the vectors $\hat{x}$, $x$ replaced by their projections on the first two coordinates, and with $\varepsilon_0$, $\varepsilon_1$ changed appropriately. We obtain a free subgroup $F''\le \SA(2,D,\varepsilon_1)\le \SA(2,D,\varepsilon)$ acting on $K^2\setminus\{(\hat{x}_1,\hat{x}_2)\}$ without nontrivial fixed points. As before, $\{T_i\}_{i\in S(\varepsilon_1)}$ from \eqref{T_i} is a basis of $F''$.
    Let $F':=\iota_{2,n}(F'')$, where $\iota_{2,n}$ is defined as in \eqref{canonical_embedding}. Clearly, $F'$ is a free subgroup of infinite rank of $\SA(n,D,\varepsilon)$. We will show that $F'$ satisfies the claim of our theorem. It is not hard to check (by examining the equation \eqref{canonical_embedding}) that $F'$ acts without nontrivial fixed points on the set $Y:=K^n\setminus \left(\{(\hat{x}_1,\hat{x}_2)\}\times K^{n-2}\right)$, which contains $B[x,r]$, $B(x,r)$ and $S[x,r]$.  
    As in the proof of Proposition \ref{free_in_affine}, we have
    $$
    \lVert \iota_{2,n}(T_i)(x)-x \rVert =\lVert T_i(x_1,x_2)-(x_1,x_2) \rVert <\varepsilon_0 \lVert (x_1,x_2)-(\hat{x}_1,\hat{x}_2) \rVert \le r
    $$
    for all $i\in S(\varepsilon_1)$. Therefore, by Lemma \ref{ultrametric_lemma}, $F'$ acts on $B[x,r]$, $B(x,r)$ and $S[x,r]$.

    The final assertion of our theorem follows easily from the definition of $u_i$ in Proposition \ref{free_in_affine} and the fact that the linear part of any element of $F'$ belongs (in any case) to $\SL(n,\mathbb{Z},\varepsilon)$.
\end{proof}

\subsection{The case when $K$ is a transcendental extension of its prime field}

Let $F_2$ be the free group of reduced words over the alphabet $\{a,b, a^{-1},b^{-1}\}$ and $R$ some commutative ring with identity. For any pair $A,B\in\SL(2,R)$, the mapping $a\mapsto A, b\mapsto B$ extends uniquely to a homomorphism $\rho_{A,B}\colon F_2 \to \SL(2,R)$, called a {\em representation} of $F_2$. If such a homomorphism is injective, the representation is called {\em faithful}. The function $\chi_{A,B}\colon F_2 \to R$, $\chi_{A,B}(g):=\tr(\rho_{A,B}(g))$, where $\tr$ stands for the matrix trace, is called the {\em character} of a representation $\rho_{A,B}$. It follows from well-known properties of the trace that $\chi_{A,B}(ghg^{-1})=\chi_{A,B}(h)$ for $g,h\in F_2$, so $\chi_{A,B}$ is constant on conjugacy classes in $F_2$. Hence, it induces a function $\tilde{\chi}_{A,B}\colon \Conj{F_2} \to R$, where $\Conj{F_2}$ is the set of all conjugacy classes in $F_2$.

Obviously, if $\chi_{A,B}(g)\neq 2$ for $g\neq 1_{F_2}$, then $\rho_{A,B}$ has trivial kernel, so it must be faithful. It turns out that if $R\subseteq K$ for a field $K$ with $\ch{K}>0$, the converse also holds, as we briefly explain in the following remark.

\begin{uw}\label{p>0}
    Let $K$ be a field with $\ch{K}=p>0$ and $A,B\in \SL(2,K)$ be such that $\rho_{A,B}$ is faithful. We claim that $\chi_{A,B}(g)\neq 2$ for $g\neq 1_{F_2}$. Indeed, if $\chi_{A,B}(g)=2$, we can see from \eqref{charpol} that $1$ is an eigenvalue of $\rho_{A,B}(g)$.
    Hence, $\rho_{A,B}(g)$ has the matrix
    $
         \left(\begin{array}{cc}
                1 & a \\
                0 & 1
            \end{array}\right)
    $, for some $a\in K$, in an appropriate basis of $K^2$.
    By direct multiplication and using the assumption on characteristic, we obtain
    $\rho_{A,B}(g^p)=I_2$.
    Since $\rho_{A,B}$ is injective, $g^p=1_{F_2}$, so $g=1_{F_2}$ because $F_2$ is torsion-free.
\end{uw}

\begin{wn}\label{corp>0}
    If $\ch{K}>0$, then any free subgroup of rank two in $\SL(2,K)$ acts on $K^2 \setminus \{\mathbf{0}\}$ without nontrivial fixed points.
\end{wn}

It is known that, for a given $w\in F_2$, the trace of $\rho_{A,B}(w)$ depends in a polynomial manner on the three values: $\tr A$, $\tr B$, $\tr AB$. This fact was first discovered by Fricke and Klein in \cite{FK} for matrices with real entries, using reasoning based on arguments in non-Euclidean geometry. Horowitz gave a constructive, algebraic proof of a stronger result \cite[Theorem 3.1]{Hor} for any free group of finite rank and for any commutative ring $R$ of characteristic zero with identity. Traina obtained a more explicit formula describing the polynomial relation  between  $\tr(\rho_{A,B}(w))$ and $\tr A$, $\tr B$, $\tr AB$. However, his paper \cite{Tr} concerns only matrices in $\SL(2,\mathbb{C})$ and some proofs are omitted.

We are going to formulate and prove an analogical result in the case of representations of $F_2$ in $\SL(2,K)$ for any field $K$ (of arbitrary characteristic). We will need an explicit construction of some polynomials of three variables $X,Y,Z$ over $\mathbb{P}$, the minimal subring of $K$ containing $1$, called also the {\em prime ring} of $K$. If $\ch{K}=0$, we identify $\mathbb{P}$ with $\mathbb{Z}$, otherwise $\mathbb{P}=\mathbb{F}_p:=\{0,1,\dots,p-1\}$ is the prime field of $K$ and is isomorphic to $\mathbb{Z}/p\mathbb{Z}$, where $p=\ch{K}$.

We are going to define a function $\Phi\colon \Conj{F_2} \to \mathbb{P}[X,Y,Z]$. For this purpose, we need some canonical way of choosing a representative of a given conjugacy class in $F_2$ (see \cite[Proposition 2.14]{LS} for a solution of the conjugacy problem for free groups). After conjugating by a suitable element, we can find in any class $W\in\Conj{F_2}$ a representative $w\in F_2$ of one the following forms:
\begin{equation}\label{postkan}
    w=1_{F_2}, w=a^{n_1}, w=b^{m_1}\; \textrm{or }
    w=a^{n_1}b^{m_1}\dots a^{n_k}b^{m_k} \;\textrm{for }  k\geq 1,
\end{equation}
where $n_i, m_i\in \mathbb{Z}\setminus \{0\}$ for $i=1,\dots, k$. If $W$ is not one of $[1_{F_2}]$, $[a^{n_1}]$, $[b^{m_1}]$, let us call each subword $s_i:=a^{n_i}b^{m_i}$, $i=1,\dots,k$, of $w$ a {\em syllable}. Then, a representative $w\in W$ of the form \eqref{postkan} is determined up to a cyclic permutation of syllables (the possible representatives vary only in the choice of an initial syllable). To avoid ambiguity, let us choose $w$ so that the vector $(\max\{\lvert n_1 \rvert, \lvert m_1 \rvert\}, \dots, \max\{\lvert n_k \rvert, \lvert m_k \rvert\})$ is maximal with respect to the lexicographic ordering in $\mathbb{N}^k$. If there are several such choices, then let us additionally maximize lexicographically $(n_1,m_1,\dots, n_k,m_k)\in \mathbb{Z}^{2k}$ among them. This particular $w\in W$ will be called the {\em canonical representative} of a conjugacy class $W$. The numbers
\begin{equation}\label{lengths}
    l_a(W):=\sum_{i=1}^{k} \lvert n_i \rvert, \quad l_b(W):=\sum_{i=1}^{k} \lvert m_i \rvert, \quad l(W):=l_a(W)+l_b(W)
\end{equation}
will be called the $a$-{\em length}, $b$-{\em length} and {\em length} of $W$, respectively.
We extend $l_a$, $l_b$ and $l$ to the cases $W=[1_{F_2}]$, $W=[a^{n_1}]$ or $W=[b^{m_1}]$ in the obvious way.

\begin{deff}\label{trace_poly}
Let $\Phi\colon \Conj{F_2} \to \mathbb{P}[X,Y,Z]$ be defined as follows. At the beginning, we specify $\Phi(W)$ for so called {\em special classes} $W$, namely we put
\begin{equation}\label{a,b,ab}
    \begin{array}{l}
        \Phi([1_{F_2}]):=2, \quad \Phi([a])=\Phi([a^{-1}]):=X, \quad \Phi([b])=\Phi([b^{-1}]):=Y, \\
        \Phi([ab])=\Phi([a^{-1}b^{-1}]):=Z, \quad \Phi([a^{-1}b])=\Phi([ab^{-1}]):=XY - Z, \\ \Phi([aba^{-1}b^{-1}])=\Phi([ab^{-1}a^{-1}b]):=X^2+Y^2+Z^2-XYZ-2,\\
        \Phi([aba^{-1}b])=\Phi([ab^{-1}a^{-1}b^{-1}]):=-X^2-Z^2+XYZ+2.
    \end{array}
\end{equation}

We will use the recursion on $l(W)$, $W\in \Conj{F_2}$. Notice that the cases $l(W)\in \{0,1\}$ are contained in \eqref{a,b,ab}.
Assume that $n\geq 2$ and we have already defined $\Phi$ for all the classes of length less than $n$. Let $l(W)=n$ and $w$ be the canonical representative of $W$ of the form \eqref{postkan}. If $w$ satisfies $\lvert n_1 \rvert\geq 2$, we can write $w=a^{n_1}u$, where $u=b^{m_1}\dots a^{n_k}b^{m_k}$ or $u=1_{F_2}$. Let us define
\begin{equation}\label{rek2}
        \Phi(W)=\Phi([a^{n_1}u]):=X \,\Phi ([a^{n_{1}-\mathrm{sgn}(n_1)}u])
        - \Phi([a^{n_{1}-2\,\mathrm{sgn}(n_1)}u]).
\end{equation}
If $\lvert m_1 \rvert\ge 2$, but $\lvert n_1 \rvert=1$ or $w=b^{m_1}$, we similarly write
$[w]=[b^{m_1}u]$, where $u=a^{n_2}b^{n_2}\dots a^{n_k}b^{m_k} a^{n_1}$ or $u=1_{F_2}$, and define
\begin{equation}\label{rek3}
        \Phi(W)=\Phi([b^{m_1}u]):=Y \,\Phi([b^{m_{1}-\mathrm{sgn}(m_1)}u])\\
        - \Phi([b^{m_{1}-2\mathrm{sgn}(m_1)}u]).
\end{equation}
Since the classes appearing on the right-hand side of the formula \eqref{rek2} or \eqref{rek3} are of lengths less than $l(W)$, the definition of $\Phi(W)$ is correct in both cases.

We have one remaining case to deal with, namely, all the exponents $n_i$, $m_i$ of $w$ in $\eqref{postkan}$ have absolute value $1$ and $W$ is not one of the special classes described by \eqref{a,b,ab}. In this case, there exist $1\leq i<j\leq k$ such that $n_i=n_j$.
Let us assume that such a pair $(i,j)$ is chosen to be as small as possible with respect to the lexicographic ordering in $\mathbb{N}^2$. We can thus write $[w]=[uv]$ for $u=s_i\dots s_{j-1}$, $v=s_j \dots s_k s_1\dots s_{i-1}$. Let us define

\begin{equation}\label{case+-1}
    \Phi(W)=\Phi([uv]):=\Phi([u]) \,\Phi([v]) - \Phi([u^{-1}v]).
\end{equation}
The definition is correct since $l([u])<l(W)$, $l([v])<l(W)$ and, by virtue of $n_i=n_j$,
there is a reduction in the product $u^{-1}v$, hence $l([u^{-1}v])<l(W)$.
The definition of $\Phi\colon \Conj{F_2} \to \mathbb{P}[X,Y,Z]$ is now complete.
\end{deff}

Let us write down some elementary properties of the trace of matrices in $\SL(2,K)$. They can be verified by a direct calculation \cite[p. 338]{FK} or by applying the Cayley--Hamilton theorem \cite[Chapter V: Lemma 18]{Baumslag}.
\begin{lem}\label{trace_identities}
    For any field $K$ and any $A, B\in\SL(2,K)$, the following identities hold:
    \begin{enumerate}
        \item $\tr A^{-1}=\tr A$.
        \item $\tr{A^{-1}B}=\tr A\cdot\tr B-\tr AB$.
    \end{enumerate}
\end{lem}

Using the above-stated lemma and the definition of $\Phi$, we can inductively prove the following theorem.

\begin{tw}\label{char}
    Let $K$ be a field. For any $A,B\in\SL(2,K)$ and $w\in F_2$, we have the following equality:
    \begin{equation}\label{wielslad}
        \chi_{A,B}(w)=\tr(\rho_{A,B}(w))=\Phi([w])(\tr A,\tr B,\tr AB).
    \end{equation}
\end{tw}

\begin{proof}
    First, we will show \eqref{wielslad} for $w$ such that $[w]$ is one of the special classes from \eqref{a,b,ab}. The cases $w=1_{F_2}$, $w=a$, $w=b$, $w=ab$ are obvious. Using the part (1) of Lemma \ref{trace_identities}, we get \eqref{wielslad} for $w=a^{-1}$ and $w=b^{-1}$. Since
    $\tr A^{-1}B^{-1}=\tr B^{-1}A^{-1}=\tr AB$, the result follows for $w=a^{-1}b^{-1}$. From the part (2) of Lemma \ref{trace_identities}, we obtain
    \[\tr A^{-1}B=\tr A\cdot \tr B-\tr AB=\Phi([a^{-1}b])(\tr A,\tr B,\tr AB).\]
    Moreover, $\tr AB^{-1}=\tr B^{-1}A=\tr A^{-1}B$, which yields the result for $w=a^{-1}b$ and $w=ab^{-1}$. For the case of $w=aba^{-1}b^{-1}$, we calculate
    \begin{equation*}
    \begin{split}
        \tr ABA^{-1}B^{-1}&=\tr A\cdot \tr BA^{-1}B^{-1}-\tr A^{-1}BA^{-1}B^{-1}\\
        &=(\tr A)^2 - (\tr A^{-1}B\cdot\tr A^{-1}B^{-1}-\tr B^{-2})\\
        &=(\tr A)^2+\tr B^2 - (\tr A\cdot\tr B-\tr AB)\cdot\tr AB\\
        &=(\tr A)^2+(\tr B)^2-2+(\tr AB)^2 - \tr A\cdot\tr B\cdot\tr AB,
    \end{split}
    \end{equation*}
    hence $\tr ABA^{-1}B^{-1}=\Phi([aba^{-1}b^{-1}])(\tr A,\tr B,\tr AB)$.
    Moreover, $\tr AB^{-1}A^{-1}B=\tr BAB^{-1}A^{-1}=\tr ABA^{-1}B^{-1}$, so \eqref{wielslad} holds for $w=ab^{-1}a^{-1}b$ as well.
    To deal with the remaining two-syllable cases simultaneously, let $\varepsilon\in \{1,-1\}$. We calculate
    \begin{equation*}
    \begin{split}
        \tr AB^{\varepsilon}A^{-1}B^{\varepsilon}&=\tr AB^{\varepsilon}A^{-1}\cdot \tr B^{\varepsilon}-\tr AB^{-\varepsilon}A^{-1}B^{\varepsilon}\\
        &=(\tr B)^2 - \tr ABA^{-1}B^{-1}\\
        &=-(\tr A)^2-(\tr AB)^2 + \tr A\cdot\tr B\cdot\tr AB+2,
    \end{split}
    \end{equation*}
    hence $\tr AB^{\varepsilon}A^{-1}B^{\varepsilon}=\Phi([ab^{\varepsilon}a^{-1}b^{\varepsilon}])(\tr A,\tr B,\tr AB)$.

    We complete the proof by induction on $l([w])$. Assume that $n\ge 2$ and \eqref{wielslad} holds for any $v\in F_2$ with $l([v])<n$. Let $l([w])=n$ and $W:=[w]$. We may assume that $w$ is the canonical representative of $W$ of the form \eqref{postkan}.
    If $\lvert n_1 \rvert\ge 2$, we denote $U:=\rho_{A,B}(u)$, where $u$ is such as in \eqref{rek2}. Then, by \eqref{rek2} and the inductive assumption, we obtain the relation
    \begin{equation}\label{traces_in_proof}
        \Phi([w])(\tr A,\tr B,\tr AB) =  \tr A \cdot \tr A^{n_1 - \mathrm{sgn}(n_1)}U- \tr A^{n_1 - 2\mathrm{sgn}(n_1)}U.
    \end{equation}
    Using Lemma \ref{trace_identities}, we conclude that the right-hand side of \eqref{traces_in_proof} is equal to $\tr A^{n_1}U=\tr(\rho_{A,B}(w))$.

    If $\lvert m_1 \rvert\geq 2$, but $\lvert n_1 \rvert=1$ or $w=b^{m_1}$, we proceed similarly using \eqref{rek3}, the inductive assumption and Lemma \ref{trace_identities}.

    If $\Phi(W)$ is defined by \eqref{case+-1}, we put $U:=\rho_{A,B}(u)$, $V:=\rho_{A,B}(v)$. As before, we obtain
    \begin{equation*}
    \begin{split}\Phi([w])(\tr A,\tr B,\tr AB) &=\tr U\cdot \tr V - \tr U^{-1}V\\&=\tr UV=\tr(\rho_{A,B}(w)),
    \end{split}
    \end{equation*}
    which finishes the proof.
\end{proof}

In our next theorem, we prove an interesting property of polynomials obtained from $\Phi([w])$ after the substitution of three rational functions satisfying certain conditions. We use an auxiliary valuation on $K(X)$.

\begin{tw}\label{degree_of_psi}
    Let $K$ be a field and $K(X)$ the field of rational functions over $K$. Let $K(X)$ be equipped with a non-Archimedean valuation $\lvert\blank\rvert$ whose restriction to the subfield $K\subseteq K(X)$ is trivial. If $f,g,h\in K(X)$ satisfy
    \begin{equation}\label{fgh}
        \lvert f \rvert>1, \quad \lvert g \rvert>1, \quad \lvert h \rvert=\lvert fg \rvert=\lvert fg-h \rvert,
    \end{equation}
    then, for any $W\in\Conj{F_2}$, $W\ne [1_{F_2}]$, the rational function $\Psi(W):=\Phi(W)(f, g, h)\in K(X)$ satisfies
    \begin{equation}\label{val_of_psi_formula}
        \lvert \Psi(W) \rvert = \lvert f \rvert^{l_a(W)} \cdot \lvert g \rvert^{l_b(W)} >1,
    \end{equation}
    where $l_a$, $l_b$ are defined by \eqref{lengths}. In particular, $\Psi(W)\not\in K$ for $W\ne [1_{F_2}]$.
\end{tw}

\begin{proof}
    First, we will show \eqref{val_of_psi_formula} for special classes $W\ne [1_{F_2}]$. It is clear for
    $W\in \{[a], [a^{-1}], [b], [b^{-1}]\}$.
    By \eqref{fgh}, we also have
    \begin{align*}
        \lvert \Psi([ab]) \rvert&=\lvert \Psi([a^{-1}b^{-1}]) \rvert=\lvert h \rvert=\lvert f \rvert^1\cdot \lvert g \rvert^1,\\
        \lvert \Psi([a^{-1}b]) \rvert&=\lvert \Psi([ab^{-1}]) \rvert=\lvert fg-h \rvert=\lvert f \rvert^1\cdot \lvert g \rvert^1,\\
        \lvert \Psi([aba^{-1}b^{-1}]) \rvert&=\lvert \Psi([ab^{-1}a^{-1}b]) \rvert=\lvert f^2 + g^2 + h^2 - fgh -2 \rvert\\
        &=\lvert f^2 + g^2 - 2 - (fg-h)h \rvert,\\
        \lvert \Psi([aba^{-1}b]) \rvert&=\lvert \Psi([ab^{-1}a^{-1}b^{-1}]) \rvert=\lvert -f^2-h^2+fgh+2 \rvert\\
        &=\lvert (fg-h)h - (f^2-2) \rvert.
    \end{align*}
    Since
    \[\lvert f^2+g^2-2 \rvert\le \max\{\lvert f \rvert^2, \lvert g \rvert^2\}<\lvert f \rvert^2\cdot \lvert g \rvert^2=\lvert fg \rvert^2=\lvert fg-h\lvert\blank\rvert h \rvert=\lvert (fg-h)h \rvert\]
    and $\lvert f^2-2 \rvert=\lvert f \rvert^2<\lvert (fg-h)h \rvert$, we obtain
    \begin{align*}
        \lvert \Psi([aba^{-1}b^{-1}]) \rvert&=\lvert \Psi([ab^{-1}a^{-1}b]) \rvert=\lvert \Psi([aba^{-1}b]) \rvert=\lvert \Psi([ab^{-1}a^{-1}b^{-1}]) \rvert\\
        &=\lvert (fg-h)h \rvert=\lvert f \rvert^2 \cdot \lvert g \rvert^2,
    \end{align*}
    as desired.

    Assume that $n\geq 2$ and \eqref{val_of_psi_formula} holds for all classes of length less than $n$. Let $l(W)=n$ and $w$ be the canonical representative of $W$ of the form \eqref{postkan}. Assume that $\lvert n_1 \rvert\ge 2$ and let us abbreviate
    the elements of $F_2$ on the right-hand side of \eqref{rek2} by $w_1$ and $w_2$. Thus we have $\Psi(W)=f \,\Psi([w_1]) - \Psi([w_2])$. Examining the $a$-length and $b$-length of $[w_1]$ and $[w_2]$, we get $l_a([w_1])=l_a(W)-1$, $l_b([w_1])=l_b(W)$ and $l_a([w_2])\le l_a(W)-2$, $l_b([w_2])\le l_b(W)$. We deduce that $l([w_1])<l(W)$ and $l([w_2])< l(W)$. By the inductive assumption,
    \begin{align*}
    \lvert \Psi([w_1]) \rvert&=\lvert f \rvert^{l_a([w_1])}\cdot \lvert g \rvert^{l_b([w_1])}=\lvert f \rvert^{l_a(W)-1} \cdot \lvert g \rvert^{l_b (W)},\\
    \lvert \Psi([w_2]) \rvert&=\lvert f \rvert^{l_a([w_2])}\cdot \lvert g \rvert^{l_b([w_2])}\le \lvert f \rvert^{l_a(W)-2} \cdot \lvert g \rvert^{l_b (W)}.
    \end{align*}
    Therefore,
    \[\lvert f \, \Psi([w_1]) \rvert =\lvert f \rvert^{l_a(W)} \cdot \lvert g \rvert^{l_b (W)} >\lvert \Psi([w_2]) \rvert,\] 
    so $\lvert \Psi(W) \rvert=\lvert f \rvert^{l_a(W)}\cdot \lvert g \rvert^{l_b(W)}$.

    If $\lvert m_1 \rvert\geq 2$, but $\lvert n_1 \rvert=1$ or $w=b^{m_1}$, we proceed analogically using \eqref{rek3}.

    If $\Phi(W)$ is defined by \eqref{case+-1}, we have
    $l_a(W)=l_a([u])+l_a([v])$, $l_b(W)=l_b([u])+l_b([v])$. Hence, by the inductive assumption,
    \[\lvert \Psi([u])\,\Psi([v]) \rvert=\lvert f \rvert^{l_a([u])}\cdot\lvert g \rvert^{l_b([u])}\cdot\lvert f \rvert^{l_a([v])}\cdot\lvert g \rvert^{l_b([v])}=\lvert f \rvert^{l_a(W)}\cdot \lvert g \rvert^{l_b(W)}.\]
    Because of a reduction in the product $u^{-1}v$, we have
    $l_a([u^{-1}v])< l_a(W)$, $l_b([u^{-1}v]) \le l_b(W)$, so
    \[\lvert \Psi([u^{-1}v]) \rvert=\lvert f \rvert^{l_a([u^{-1}v])}\cdot \lvert g \rvert^{l_b([u^{-1}v])}<\lvert f \rvert^{l_a(W)}\cdot \lvert g \rvert^{l_b(W)}=\lvert \Psi([u])\,\Psi([v]) \rvert.\]
    We finally get $\lvert \Psi(W) \rvert=\lvert \Psi([u])\,\Psi([v]) \rvert$, from which \eqref{val_of_psi_formula} follows.
\end{proof}

Let us formulate an important corollary from Theorem \ref{degree_of_psi}.

\begin{wn}\label{wnprz}
    Let $K$ be a field and $K_0$ its prime field. Let $K_0(X)$ be equipped with a non-Archimedean valuation $\lvert\blank\rvert$ trivial on $K_0$ and let $f,g,h\in K_0(X)$ satisfy \eqref{fgh}. If $t\in K$ is transcendental over $K_0$, then $\Phi(W)(f(t),g(t),h(t)) \neq 2$ for all classes $W\in\Conj{F_2}$ distinct from $[1_{F_2}]$.
\end{wn}

\begin{proof}
    Suppose that $\Phi(W)(f(t),g(t),h(t))=2$ for some $W\in \Conj{F_2}$, $W\ne [1_{F_2}]$, and let $\Psi(W)\in K_0(X)$ be as in Theorem \ref{degree_of_psi}. Then $t$ would satisfy the equation $\Psi(W)(t)-2=0$. However, by virtue of Theorem \ref{degree_of_psi}, the rational function $\Psi(W)$ is not constant, neither is $\Psi(W)-2$. Therefore $t$ would be algebraic over $K_0$, which leads to a contradiction.
\end{proof}

Let us introduce a family of nontrivial non-Archimedean valuations on $K(X)$, each being trivial when restricted to $K$. Namely, for $q$ being an irreducible polynomial in $K[X]$, we define $\lvert\blank\rvert_q$ on $K(X)$ by putting $\lvert 0 \rvert_q:= 0$ and
\begin{equation}\label{irr_pol_valuation}
    \left\lvert q^k \frac{f_1}{f_2}\right \rvert_q:=2^{-k} \quad \text{for } k\in\mathbb{Z}, \; q\nmid f_1, q\nmid f_2, \; f_1, f_2\in K[X],\; f_2\ne 0.
\end{equation}

\begin{lem}\label{good_rational_functions}
    Let $K_0$ be a prime field and $q\in K_0[X]$ be the irreducible polynomial specified as below:
    \begin{equation}\label{q}
          q:=\begin{cases}
         2 X + 1 & \text{if } K_0 \ne \mathbb{F}_2,\\
         X^2+X+1 & \text{if } K_0 = \mathbb{F}_2.
        \end{cases}
    \end{equation}
    Then, the conditions \eqref{fgh} are satisfied for the valuation $\lvert\blank\rvert_q$ on $K_0(X)$ and $f,g,h\in K_0(X)$ defined by $f=g:=q+q^{-1}$, $h:=q^{-2}(q^4+X+1)$.
\end{lem}

\begin{proof}
We have $\lvert f \rvert_{q}=\lvert g \rvert_{q}=2$ and $\lvert h \rvert_q=4$ since $q\nmid X+1$. We also have $\lvert fg-h \rvert_q=\lvert 2-q^{-2}X \rvert_q=\lvert q^{-2}X \rvert_q=4$ since $q\nmid X$.
\end{proof}

Now we are ready to prove results analogical to those in Proposition \ref{prop_SL(2,Z)} and Theorem \ref{Sato-like}, but in the case of $K$ being a transcendental extension of its prime field $K_0$.

\begin{tw}\label{pos_char}
    Let $(K,\lvert\blank\rvert)$ be a non-Archimedean valued field, $D$ its ring of integers and $K_0$ its prime field. Let $q\in K_0[X]$ be chosen according to \eqref{q}. If $K$ is a transcendental extension of $K_0$, then there exists $t\in D$ such that the matrices $A,B$ given by
    \begin{equation}\label{AB2}
        A:=\left(
        \begin{array}{cc}
                q(t^2) & t            \\
                0       & (q(t^2))^{-1}
            \end{array}
        \right),\quad
        B:=\left(
        \begin{array}{cc}
                q(t^2)                   & 0            \\
                (q(t^2))^{-2}\cdot t & (q(t^2))^{-1}
            \end{array}
        \right)
    \end{equation}
    form a basis of a free subgroup $\left<A,B\right>\le \SL(2,D)$ that acts on on $K^2\setminus\{\mathbf{0}\}$, and (for any $r>0$) on $B[\mathbf{0},r]\setminus\{\mathbf{0}\}$ and $S[\mathbf{0},r]$ without nontrivial fixed points.

    Moreover, if the valuation $\lvert\blank\rvert$ is nontrivial, then, for any $\varepsilon\in(0,1]$, $t$ can be chosen so that $A,B\in \SL(2,D,\varepsilon)$.
\end{tw}

\begin{proof}
    If the valuation $\lvert\blank\rvert$ is trivial, let $t\in K$ be any element transcendental over $K_0$; hence, $q(t^2)\ne 0$ and $A,B$ are well-defined elements of $\SL(2,K)=\SL(2,D)$.
    If $\lvert\blank\rvert$ is nontrivial, then, for any $\varepsilon\in(0,1]$, we can always find an element $t\in K$ transcendental over $K_0$ and satisfying  $\lvert t \rvert<\varepsilon$. We then have $\lvert q(t^2)-1 \rvert\le \lvert t^2 \rvert<\varepsilon$ and $\lvert q(t^2) \rvert=1$. Since $q(t^2)$ belongs to the coset $1+B(0,\varepsilon)$ modulo the ideal $B(0,\varepsilon)$ in $D$, so does $(q(t^2))^{-1}$. Therefore, $A,B\in \SL(2,D,\varepsilon)$.

    Let $f,g,h\in K_0[X]$ be defined as in Lemma \ref{good_rational_functions}. By a direct calculation,
    $\tr A=f(t^2)$, $\tr B=g(t^2)$ and
    $\tr AB=(q(t^2))^{2}+(q(t^2))^{-2}+(q(t^2))^{-2}\,t^2=h(t^2)$.

    Let $\rho_{A,B}\colon F_2 \to \SL(2,D)$ be the representation of $F_2$ specified by $A,B$ from \eqref{AB2}. Assume that $w\in F_2$, $w\ne 1_{F_2}$. From Theorem \ref{char} we get
    \[\tr(\rho_{A,B}(w))=\Phi([w])(\tr A,\tr B,\tr AB)=\Phi([w])(f(t^2),g(t^2),h(t^2)).\]
    Since $t^2$ is transcendental over $K_0$, we deduce from Corollary \ref{wnprz} that $\tr(\rho_{A,B}(w))\ne 2$, which implies that $\rho_{A,B}(w)\ne I_2$ and (by virtue of Lemma \ref{trace}) $\rho_{A,B}(w)$ has no fixed points in $K^2\setminus \{\mathbf{0}\}$.
\end{proof}

\begin{tw}\label{affine_general}
    Let $(K,\lvert\blank\rvert)$ be a non-Archimedean nontrivially valued field that is a transcendental extension of its prime field $K_0$. Let $n\ge 2$, $\hat{x}=(\hat{x}_1,\hat{x}_2,\dots,\hat{x}_n)\in K^n$, $x=(x_1,x_2,\dots,x_n)\in K^n$, $(x_1,x_2)\ne (\hat{x}_1,\hat{x}_2)$, $0<r<\lVert (x_1,x_2)-(\hat{x}_1,\hat{x}_2) \rVert $ and $\varepsilon\in(0,1]$. Then, there exists a free subgroup $F'\le \SA(n,D,\varepsilon)$ of rank two acting on $K^n\setminus \left(\{(\hat{x}_1,\hat{x}_2)\}\times K^{n-2}\right)$, the balls $B[x,r]$, $B(x,r)$, and the sphere $S[x,r]$ without nontrivial fixed points.

    Moreover, $F'$ can be chosen so that $F'\le \SL(n,D,\varepsilon)$ if $(\hat{x}_1,\hat{x}_2)=(0,0)$.
\end{tw}

\begin{proof}
    The proof is essentially the same as that of Theorem \ref{Sato-like}. The only significant difference is that we take for $F(\varepsilon_1)$ the free group with basis $\{A,B\}$, where $A$, $B$ are the matrices \eqref{AB2} with $t$ chosen so that $A,B\in \SL(2,D,\varepsilon_1)$.
\end{proof}

\section{The Banach--Tarski paradox for subsets of $K^n$}

In this section, by $R$ we will mean $\mathbb{Z}$ if $\ch{K}\ne \ch{k}$, and $D$ otherwise.

Clearly, any subset $A$ of $K^n$ containing $\mathbf{0}$ is not paradoxical with respect to any group of linear isometries of $K^n$.
For balls in $K^n$ not containing $\mathbf{0}$, we obtain the following positive result.

\begin{tw}\label{balls_without_0}
    Let $(K,\lvert\blank\rvert)$ be a non-Archimedean nontrivially valued field, $k$ its residue field, $\varepsilon\in (0,1]$ and $n\geq 2$. Then any ball $B$ in $K^n$ not containing $\mathbf{0}$ is $\SL(n,R,\varepsilon)$-paradoxical using $4$ pieces.
\end{tw}

\begin{proof}
    We either have $\ch{K}=0\ne \ch{k}$, or $\ch{K}=\ch{k}$. In the latter case, $\lvert\blank\rvert$ is trivial on the prime field $K_0$ of $K$ and any element $t\in K^*$ with $\lvert t \rvert\neq 1$ is transcendental over  $K_0$, so $K$ is a transcendental extension of $K_0.$ 

    Assume that $B$ is a closed ball $B[x,r]$ in $K^n$ not containing $\mathbf{0}$. Then $\lVert x \rVert >r$, so $\lvert x_i \rvert>r$ for some $1\leq i \leq n.$ 
    
    If $i=1$, we put $\hat{x}:=\mathbf{0}$. From Theorem \ref{Sato-like} or \ref{affine_general}, it follows that there exists a non-Abelian free subgroup $F'\le \SL(n,R,\varepsilon)$ acting on $B[x,r]$ without nontrivial fixed points, which yields $\SL(n,R,\varepsilon)$-paradoxicality of $B[x,r]$ using $4$ pieces.

    Assume that $i>1$. Let $\gamma$ be a permutation of the set $\{1, \ldots, n\}$ with $\gamma (1)=i$ and $P\in \GL(n,R)$ with $P(e_j)=e_{\gamma (j)}$ for any $1\leq j \leq n.$
    If $V\in \SL(n,R,\varepsilon)$ and $M=PVP^{-1}$, then $m_{k,l}=v_{\gamma^{-1}(k),\gamma^{-1}(l)}$ for all $1\leq k,l\leq n$, so $M\in \SL(n,R,\varepsilon)$. Put $z=P^{-1}x.$ Then $\lvert z_1 \rvert=\lvert x_i \rvert>r$ and $P(B[z,r])=B[x,r]$, so $\mathbf{0}\not \in B[z,r].$ Thus, $B[z,r]$ is $SL(n,R,\varepsilon)$-paradoxical using $4$ pieces, so $B[x,r]=P(B[z,r])$ is $SL(n,R,\varepsilon)$-paradoxical using $4$ pieces too.

     The same holds for any open ball $B=B(x,r)$ in $K^n$ not containing $\mathbf{0}$ because the assumption $\lVert x \rVert \geq r$ is then sufficient by Remark \ref{remark_open_ball}.
\end{proof}

As in the case of balls, spheres not containing $\mathbf{0}$ are also $\SL(n,R,\varepsilon)$-paradoxical, but the number of pieces involved is expressed in a more complicated manner.

\begin{tw}\label{spheres}
    Let $(K,\lvert\blank\rvert)$ be a non-Archimedean nontrivially valued field, $k$ its residue field, $\varepsilon\in (0,1]$, $r\in \lvert K^{*} \rvert$ and $n\ge 2$. Then any sphere $S[x,r]$ in $K^n$ not containing $\mathbf{0}$ is $\SL(n,R,\varepsilon)$-paradoxical using $4$ pieces if $\lVert x \rVert >r$, and $5-(-1)^n$ pieces if $\lVert x \rVert <r$.
\end{tw}

\begin{proof}
    First, notice that if $\lVert x \rVert >r$, we can proceed exactly as in the proof of Theorem \ref{balls_without_0}. As before, Theorem \ref{Sato-like} or \ref{affine_general} provides the existence of a non-Abelian free subgroup $F'\le \SL(n,R,\varepsilon)$ acting on $S[x,r]$ without nontrivial fixed points.

    In the case $\lVert x \rVert <r$, we have $S[x,r]=S[\mathbf{0},r]$ by the isosceles property. Hence, we assume from now on that $x=\mathbf{0}$.
    Let us denote by $F(\varepsilon)$ a free subgroup of $\SL(2,R,\varepsilon)$ obtained from Proposition \ref{prop_SL(2,Z)} or Theorem \ref{pos_char}, depending on whether $\ch{K}$ and $\ch{k}$ are distinct or equal.

    Assume that $n=2k$, $k\ge 1$. We can treat $K^n$ as the product of $k$ copies of $K^2$. Let us define an embedding $\varphi\colon F(\varepsilon)\to \SL(n,R,\varepsilon)$ by putting
    \[\varphi(g)(y_1,\dots,y_{2k}):=\left(g(y_1,y_2),\dots,g(y_{2k-1},y_{2k})\right)\in (K^2)^k \cong K^n\]
    for $g\in F(\varepsilon)$, $(y_1,\dots,y_{2k})\in K^n$, i.e., $\varphi(g)$ acts like $g$ on each copy of $K^2$.
    Since $F(\varepsilon)$ acts without nontrivial fixed points on $K^2\setminus \{0_2\}$, so does $\varphi(F(\varepsilon))$ on $K^n\setminus \{\mathbf{0}\}$, hence also on $S[\mathbf{0},r]$. Therefore, $S[\mathbf{0},r]$ is $\SL(n,R,\varepsilon)$-paradoxical using $4$ pieces.

    Assume now that $n=2k+1$, $k\ge 1$. Let us partition $S[\mathbf{0},r]$, regarded as a subset of $K^{2k}\times K\cong  K^{n}$, into two disjoint sets. Namely, we put
    \[A_1:=S[0_{2k},r]\times B[0_1,r], \qquad A_2:=B(0_{2k},r)\times S[0_1,r].\] It is easy to see that $A_1$ and $A_2$ are $\SL(n,R,\varepsilon)$-invariant.
    
    Let us define two embeddings $\varphi_1,\varphi_2 \colon F(\varepsilon)\to \SL(n,R,\varepsilon)$ as follows
    \begin{equation*}
        \begin{split}
            \varphi_1(g)(y_1,\dots,y_{n}):=\left(g(y_1,y_2),\dots, g(y_{2k-1},y_{2k}),y_{n}\right)\in (K^2)^k \times K \cong K^n,\\
            \varphi_2(g)(y_1,\dots,y_{n}):=\left(y_1,g(y_2,y_3),\dots, g(y_{2k},y_{2k+1})\right)\in K\times (K^2)^k  \cong K^n,
        \end{split}
    \end{equation*}
    for $g\in F(\varepsilon)$, $(y_1,\dots,y_{n})\in K^n$.
    Our choice of $\varphi_1, \varphi_2$ guarantees that, for each $i\in \{1,2\}$, any nonidentity element of $\varphi_i(F(\varepsilon))$ has no fixed points in $A_i$. 
    Let $\{\sigma,\tau\}$ be a basis of $F(\varepsilon)$. As in $\eqref{characterization_four_pieces}$, we have the decompositions
    \begin{equation}\label{A_i_decomp}
        \begin{split}
        A_1=A_{11} \sqcup A_{12} \sqcup A_{13} \sqcup A_{14} = A_{11} \sqcup \varphi_1(\sigma)(A_{12}) = A_{13} \sqcup \varphi_1(\tau)(A_{14}),\\
        A_2=A_{21} \sqcup A_{22} \sqcup A_{23} \sqcup A_{24} = A_{21} \sqcup \varphi_2(\sigma)(A_{22}) = A_{23} \sqcup \varphi_2(\tau)(A_{24}).
        \end{split}
    \end{equation}
    Let us abbreviate $E:=A_{11}\sqcup A_{21}$, $E':=A_{13}\sqcup A_{23}$.
    Combining the equations \eqref{A_i_decomp}, we obtain
    \begin{equation}
        \begin{split}
            S[\mathbf{0},r]&=E \sqcup A_{12} \sqcup A_{22}\sqcup E' \sqcup A_{14} \sqcup A_{24}\\
            &= E \sqcup \varphi_1(\sigma)(A_{12}) \sqcup \varphi_2(\sigma)(A_{22})\\
            &= E' \sqcup \varphi_1(\tau)(A_{14}) \sqcup \varphi_2(\tau)(A_{24}),
        \end{split}
    \end{equation}
    which yields an $\SL(n,R,\varepsilon)$-paradoxical decomposition of $S[\mathbf{0},r]$ using $6$ pieces.
\end{proof}

\begin{wn}\label{K^n_minus_0}
    Let $(K,\lvert\blank\rvert)$ be a non-Archimedean, nontrivially valued field, $\varepsilon\in (0,1]$ and $n\ge 2$. Then, the union of any family of nonempty spheres in $K^n$ centered at $\mathbf{0}$ is $\SL(n,R,\varepsilon)$-paradoxical using $5-(-1)^n$ pieces. In particular, the claim holds for $K^n\setminus\{\mathbf{0}\}$ and $D^n\setminus \{\mathbf{0}\}$.
\end{wn}

\begin{proof}
    Let $Z\subseteq \lvert K^{*} \rvert$.
    Since the group elements witnessing the $\SL(n,R,\varepsilon)$-paradoxical decomposition of $S[\mathbf{0},r]$ in the proof of Theorem \ref{spheres} do not depend on $r\in Z$, the corresponding pieces from those decompositions can be gathered together, yielding a desired $\SL(n,R,\varepsilon)$-paradoxical decomposition of $\bigsqcup_{r\in Z} S[\mathbf{0},r]$. 
\end{proof}

Now we will prove that balls and spheres containing $\mathbf{0}$ admit a paradoxical decomposition with respect to a certain group of affine isometries of $K^n$.

\begin{tw}\label{sets_with_0}
    Let $(K,\lvert\blank\rvert)$ be a non-Archimedean, nontrivially valued field, $\varepsilon\in (0,1]$ and $n\geq 2.$ Then any ball and any sphere in $K^n$ containing $\mathbf{0}$ is $\SA(n,D,\varepsilon)$-paradoxical using $4$ pieces.
\end{tw}

\begin{proof}
    Let $x\in K^n$ and $r>0$. Let us choose a point $\hat{x}\in K^n$ satisfying $\lVert (x_1,x_2)-(\hat{x}_1,\hat{x}_2) \rVert >r$. Then, by Theorem \ref{Sato-like} or \ref{affine_general}, there exists a non-Abelian free subgroup $F'\le \SA(n,D,\varepsilon)$ acting on $B[x,r]$, $B(x,r)$ and $S[x,r]$ without nontrivial fixed points. It follows that $B[x,r]$, $B(x,r)$, and $S[x,r]$ (if $r\in \lvert K^{*} \rvert$) are $\SA(n,D,\varepsilon)$-paradoxical using $4$ pieces.
\end{proof}

Our next theorem is an analog of the strong result of Banach and Tarski \cite[Theorem 24]{BT}. It follows that if $K$ is a non-Archimedean discretely valued field with finite residue field (in particular, if $K$ is a locally compact non-Archimedean valued field, e.g., $K=\mathbb{Q}_p$), then any two balls in $K^n$ are equidecomposable with respect to a certain group of isometries.

\begin{tw}\label{strong_paradox_discrete_finite_residue}
    Let $(K,\lvert\blank\rvert)$ be a non-Archimedean discretely valued field with finite residue field (e.g., a locally compact non-Archimedean valued field) and $n\geq 2$. Then any two bounded subsets of $K^n$ with nonempty interiors are $\SA(n,D,K)$-equidecomposable.
\end{tw}

\begin{proof}
    Let $G:=\SA(n,D,K)$ and $A, A'\subseteq K^n$ be bounded sets with nonempty interiors. There exists a closed ball $E:=B[\mathbf{0},r]\subseteq K^n$ containing both $A$ and $A'$. By Lemma \ref{discrete_finite_residue}, we have
    $E\subseteq g_1 A \cup \dots \cup g_m A$ for some $m\in\mathbb{N}$ and translations $g_1,\dots,g_m \in G$.

    Since $E$ is $G$-paradoxical (by Theorem \ref{sets_with_0}), an application of \cite[Corollary 10.22]{W} yields that $A\sim_G E\sim_G A'$; hence, $A\sim_G A'$.
\end{proof}

\begin{wn}
    Let $(K,\lvert\blank\rvert)$ be a non-Archimedean discretely valued field with finite residue field, e.g., $(\mathbb{Q}, \lvert\blank\rvert_p)$ or $(\mathbb{Q}_p, \lvert\blank\rvert p)$, and $n\ge 2$. Then any bounded subset of $K^n$ with nonempty interior is $\SA(n,D,K)$-paradoxical.
\end{wn}

\begin{proof}
    By the previous theorem, any bounded subset of $K^n$ with nonempty interior is $\SA(n,D,K)$-equidecomposable with the ball $B[\mathbf{0},1]$, which is $\SA(n,D,K)$-paradoxical by Theorem \ref{sets_with_0}. The assertion follows from \cite[Proposition 3.5]{W}.
\end{proof}

It turns out that the whole space $K^n$ is also paradoxical with respect to a certain group of affine isometries.

\begin{tw}\label{whole_space}
    Let $(K,\lvert\blank\rvert)$ be a non-Archimedean, nontrivially valued field and $n\geq 2$. Let $L:=\mathbb{Z}$ if $\ch{K}=0$, and $L:=D$ if $\ch{K}>0$.  Then $K^n$ is $\SA(n,L)$-paradoxical using $5$ pieces.
\end{tw}

\begin{proof}
    First, we will prove the claim for $n=2$. If $\ch{K}=0$, let $F\le \SL(2,\mathbb{Z})$ be a free group from Corollary \ref{free_in_char_0}. If $\ch{K}=p>0$, then (by \cite[Theorem 14.2]{UC}) $K$ is necessarily a transcendental extension of its prime field $\mathbb{F}_p$, so let $F\le \SL(2,D)$ be the free group from Theorem \ref{pos_char}. Let us choose any point $c\in L^2\setminus\{\mathbf{0}\}$, we then have $\mathrm{Stab}_{F}(c)=\{I_2\}$. Let $C:=\{c\}\subseteq E:=K^2\setminus\{\mathbf{0}\}$. If $\alpha$ denotes the translation along $-c$, then $\alpha \in \SA(2, L)$ and $\alpha(C)=\{\mathbf{0}\}$. By virtue of Theorem \ref{zbC}, the space $K^2=E\sqcup \alpha(C)$ is $\SA(2, L)$-paradoxical using $5$ pieces.

    Assume now that $n> 2$. Let $\iota\colon \SA(2, L) \to \SA(n, L)$ be the restriction of the canonical embedding $\iota_{2,n}$, defined by \eqref{canonical_embedding}, to $\SA(2,L)$. Let also $f\colon K^n \to K^2$ be the projection onto the first two coordinates. Since $\iota$ and $f$ satisfy the condition \eqref{gof}, we infer from Lemma \ref{lift} that $K^n$ is $\SA(n,L)$-paradoxical using $5$ pieces.
\end{proof}

Let us briefly discuss the case when the valuation on $K$ is trivial. Then $D=K$ and the linear (or affine) isometries of $K^n$ are just the linear (or affine) bijections.

\begin{tw}\label{trivial_valuation}
    Let $K$ be a field and $n\ge 2$. 
    
    If $\ch{K}=0$, then $K^n$ is $\SA(n,\mathbb{Z})$-paradoxical using $5$ pieces. 
    
    If $\ch{K}=p>0$ and $t\in K$ is transcendental over $\mathbb{F}_p$, then $K^n$ is $\SA(n,\mathbb{F}_p(t))$-paradoxical using $5$ pieces.

    If $K$ is an algebraic extension of a finite subfield, then $K^n$ is not $\GA(n,K)$-paradoxical and $K^n\setminus \{\mathbf{0}\}$ is not $\GL(n,K)$-paradoxical.
\end{tw}

\begin{proof}
    Let $n\ge 2$. If $\ch{K}=0$, let $F\le \SL(2,\mathbb{Z})$ be a free group from Corollary \ref{free_in_char_0}. If $\ch{K}=p$ and $t\in K$ is transcendental over $\mathbb{F}_p$, let $F\le \SL(2,\mathbb{F}_p(t))$ be the free group from Theorem \ref{pos_char}. We further proceed exactly as in the proof of Theorem \ref{whole_space}.

    If $K$ is an algebraic extension of a finite field, then (see Example \ref{alg_ext_of_finite}) the group $\GL(n,K)$ is amenable, which excludes the existence of a $\GL(n,K)$-paradoxical decomposition of a $\GL(n,K)$-invariant set $K^n\setminus\{\mathbf{0}$\}. The group $\GA(n,K)$ is amenable by Lemma \ref{A(n,K)_amenable}, so $K^n$ cannot be $\GA(n,K)$-paradoxical.
\end{proof}

\section*{Acknowledgements}

The author would like to express his gratitude to Professor Wies{\l}aw \'Sliwa for many inspiring discussions and helpful suggestions.

\end{document}